\documentclass[a4paper,11pt]{amsart}
 \usepackage{amssymb,amsfonts,latexsym}
 \usepackage{graphics,verbatim}
 \usepackage{graphicx}
 \usepackage{comment}
 \usepackage{upref}
 \usepackage{hyperref}
\usepackage[utf8]{inputenc} 
\usepackage[T1]{fontenc}    
\usepackage{lmodern}    
\usepackage{xcolor}        
\usepackage{xcolor}
 \usepackage{amsthm}
 \usepackage{enumitem}
 \usepackage{appendix}  
 \usepackage{mathtools}
 \newtheorem{theorem}{Theorem}[section]
 \newtheorem{proposition}{Proposition}[section]
 \newtheorem{lemma}{Lemma}[section]
 
 \newtheorem{remark}{Remark}[section]

 \newcommand{\dvg}{\mathrm{~d}V_{\mathbb B^N}}
 \mathtoolsset{showonlyrefs}
 \newcommand{\rn}{\ensuremath{\mathbb{B}^N}}

 \newcommand{\norm}[1]{\left\Vert#1\right\Vert}

 \newcommand{\be} {\begin{equation}}
 \newcommand{\ee} {\end{equation}}
 \newcommand{\bea} {\begin{eqnarray}}
 \newcommand{\eea} {\end{eqnarray}}
 \newcommand{\Bea} {\begin{eqnarray*}}
 	\newcommand{\Eea} {\end{eqnarray*}}

 \newcommand{\la} {\lambda}

 \newcommand{\noi} {\noindent}

  \newcommand{\bn}{\mathbb B^N}
   \newcommand{\Rn}{\mathbb R^N}

 \catcode`\@=11
 
 \makeatletter \@addtoreset{equation}{section} \makeatother

\begin{document}

 	\title[]
	{Existence, symmetry and regularity of ground states of a non linear choquard equation in the hyperbolic space}
 \author[Gupta]{Diksha Gupta}
 \address{D.~Gupta, Department of Mathematics, Indian Institute of Technology Delhi, Hauz Khas, New Delhi, Delhi 110016, India}
 \email{maz208233@maths.iitd.ac.in}
 
 \author[Sreenadh]{K.~Sreenadh}
	\address{K.~Sreenadh, Department of Mathematics, Indian Institute of Technology Delhi, Hauz Khas, New Delhi, Delhi 110016,  India}
	\email{sreenadh@maths.iitd.ac.in}
 	
 \begin{abstract}
In this paper, we explore the positive solutions of a nonlinear Choquard equation involving the green kernel of the fractional operator $(-\Delta_{\mathbb{B}^N})^{-\alpha/2}$ in the hyperbolic space, where $\Delta_{\mathbb{B}^N}$ represents the Laplace-Beltrami operator on $\mathbb{B}^N$, with $\alpha \in (0, N)$ and $N \geq 3$. This study is analogous to the Choquard equation in the Euclidean space, which involves the non-local Riesz potential operator. We consider the functional setting within the Sobolev space $H^1(\mathbb{B}^N)$, employing advanced harmonic analysis techniques, particularly the Helgason Fourier transform and semigroup approach to the fractional Laplacian. Moreover, the Hardy-Littlewood-Sobolev inequality on complete Riemannian manifolds, as developed by Varopoulos, is pivotal in our analysis. We prove an existence result for the problem \eqref{mainEq} in the subcritical case. Moreover, we also demonstrate that solutions exhibit radial symmetry, and establish the regularity properties.
\medskip

 	\end{abstract}
\maketitle
\tableofcontents

\begin{section}{Introduction}
In this paper, we explore the following nonlinear Choquard problem in the hyperbolic space $\mathbb{B}^{N}$:
\begin{equation}
 	\tag{$F_{\alpha, \lambda}$}\label{mainEq}
 	\left\{\begin{aligned} -\Delta_{\mathbb{B}^{N}} u \, - \, \lambda u \, &= \left[(- \Delta_{\mathbb{B}^{N}})^{-\frac{\alpha}{2}}|u|^p\right]|u|^{p-2}u,\\
 	u>0 \text{ in } \mathbb B^N &\,\,\,\text{ and }\quad  u\in \mathcal{H}_\lambda\left(\mathbb{B}^N\right),
  \end{aligned}
 	\right.
\end{equation}
where $\Delta_{\mathbb{B}^{N}}$ denotes the Laplace-Beltrami operator on $\mathbb{B}^{N}$, $\lambda \leq \frac{(N-1)^2}{4}$, $1 < p < 2^*_{\alpha} = \frac{N+\alpha}{N-2}$, $0 < \alpha < N$, $N \geq 3$, and $\mathcal{H}_\lambda\left(\mathbb{B}^N\right)$ denotes the completion of $C_0^{\infty}\left(\mathbb{B}^N\right)$ with respect to the norm mentioned in \eqref{norm}. Moreover, $(- \Delta_{\mathbb{B}^{N}})^{-\frac{\alpha}{2}}$ is the green kernel of the fractional operator, $\frac{(N-1)^2}{4}$ is the bottom of the $L^2$-spectrum of $-\Delta_{\bn}$ and $2^*_\alpha$ is the critical exponent in the context of the Hardy-Littlewood-Sobolev (HLS) inequality. 

The equation when posed in $\Rn$ bears a notable resemblance to  the Choquard equation, thoroughly explored by Vitali and Moroz in their study \cite{MS1}. The equation is given as follows
\begin{equation}
-\Delta u+u=\left(I_\alpha *|u|^p\right)|u|^{p-2} u \quad \text { in } \mathbb{R}^N, \label{ChoqRn}
\end{equation}
where $I_\alpha$ denotes the Riesz potential, $\alpha \in (0,N)$ and $p>1$. The Riesz potential $I_\alpha: \mathbb{R}^N \rightarrow \mathbb{R}$ is  defined as
\begin{equation*}
I_\alpha(x)=\frac{\Gamma\left(\frac{N-\alpha}{2}\right)}{\Gamma\left(\frac{\alpha}{2}\right) \pi^{N / 2} 2^\alpha|x|^{N-\alpha}},
\end{equation*}
with $\Gamma$ representing the Euler gamma function. To draw a connection between the equation \eqref{ChoqRn} and the one explored in this article, it's useful to view the Riesz potential as a representation of the inverse fractional Laplacian.  In mathematical physics, the Newtonian potential of a function $f \in \mathcal{S}(\mathbb{R}^N)$ is defined as
\begin{equation*}
I_2(f)(x)=\frac{1}{4 \pi^{\frac{N}{2}}} \Gamma\left(\frac{N-2}{2}\right) \int_{\mathbb{R}^N} \frac{f(y)}{|x-y|^{N-2}} d y,
\end{equation*}
where $\mathcal S(\Rn)$ denotes the Schwartz space of rapidly decreasing functions in $\Rn$. The convolution kernel $\frac{1}{4 \pi^{\frac{N}{2}}} \Gamma\left(\frac{N - 2}{2}\right) \frac{1}{|x|^{N - 2}}$ appearing in $I_2(f)$ matches the fundamental solution
\begin{equation*}
u(x)=\frac{1}{(N-2) \omega_{N-1}} \frac{1}{|x|^{N-2}}
\end{equation*}
of $-\Delta$ where $\omega_{N-1}$ is the surface area of the unit sphere in $\Rn$. Extending this idea, the generalization of the Newtonian potential indicates that for any function $f \in \mathcal{S}(\mathbb{R}^N)$, in the distributional space $S^\prime(\mathbb{R}^N)$, we have $I_{\alpha}(-\Delta)^{\alpha/2} f = (-\Delta)^{\alpha/2} I_\alpha f = f$, where $(-\Delta)^{\alpha/2}$ can be defined in several equivalent ways (see \cite{NPV,Mk}).

These partial differential equations (PDEs) aren't just fascinating from a purely mathematical perspective; they also show up in various physical scenarios. For example, they appear in models like the Einstein–Klein–Gordon and Einstein–Dirac systems \cite{GG}. In practical applications like nonlinear optics and Bose-Einstein condensates, similar equations help describe how waves or particle densities evolve when both linear and nonlinear effects are at play. The Choquard equation is particularly interesting because it deals with nonlocal interactions, meaning the behavior at any point in the system is influenced by the entire field. Particularly, the following Choquard equation $$-\Delta u+u=\left(I_2 *|u|^2\right)u \text{ in }  \mathbb{R}^3,\quad u \in H^1(\mathbb{R}^3)$$
originally stems from the work of Frohlich and Pekar on polarons. They studied how free electrons in an ionic lattice interact with photons or the polarization they cause, essentially depicting an electron interacting with its own induced hole \cite{F1,F2,F3}. In 1976, Ph. Choquard adapted this model to explain a one-component plasma, showing its broader relevance in plasma physics. Interestingly, the Choquard equation is also known as the Schrödinger-Newton equation when it combines quantum mechanics with Newtonian gravity (see \cite{P1,M1, Jones1995, Pen}).

The Choquard equation in $\mathbb{R}^N$ has been the focus of extensive research, particularly concerning the existence and qualitative properties of its solutions. In \cite{MS1}, the authors established the existence of solutions to \eqref{ChoqRn} by employing variational methods with the concentration-compactness lemma of P.L. Lions. The Hardy-Littlewood-Sobolev inequality in $\mathbb{R}^N$ (\cite{LL1}) is heavily utilized in these techniques. Furthermore, in \cite{MS2}, they examined a more generalized form of the Choquard equation. In a subsequent study in \cite{MS4}, the authors considered the equation with a potential, establishing the existence of solutions for the lower critical exponent corresponding to the Hardy-Littlewood-Sobolev inequality. Gao and Yang, in \cite{GY}, have also contributed by studying the Brezis-Nirenberg type Choquard equation within a bounded domain. For additional related works, see \cite{DPV, FQT, MZ, MS3, YZ, ZZ1, ZY, ZDZZ}.

Recent research has extensively aimed to generalize elliptic equations to non-Euclidean domains, driven by their relevance to various physical phenomena. This exploration seeks to understand how curvature impacts the behavior and properties of the solutions. We refer to \cite{BS,CK, DG1,Hh, Li1} and the references cited therein, without any claim of completeness. Notably, Sandeep and Mancini's seminal work \cite{MS} demonstrated the existence and uniqueness of finite-energy positive solutions for the homogeneous elliptic equation: 
\begin{equation}\label{hsm}
	-\Delta_{\mathbb{B}^{N}} u \, - \, \lambda u \, = \, u^{p}, \quad u \in H^{1}\left(\mathbb{B}^{N}\right),
\end{equation}where $\lambda \leq \frac{(N-1)^2}{4},$ $1 < p \leq \frac{N+2}{N-2}$ if $N \geq 3; $ $1 < p < \infty$ if $N=2.$ They demonstrated that, for the subcritical case and $p > 1$, the equation \eqref{hsm} has a positive solution if and only if $\lambda < \frac{(N-1)^2}{4}$, and that this solution is unique up to hyperbolic isometries, except potentially for $N=2$ and $\lambda > \frac{2(p+1)}{(p+3)^2}$. Following their work, studies in \cite{BS, DG1, DG2} examined the existence of sign-changing solutions, compactness, and non-degeneracy, while \cite{BFG, BGGV} focused on infinite energy solutions and their asymptotic behavior. Sandeep and Mancini also found that equation \eqref{hsm} emerges naturally when analyzing Euler-Lagrange equations related to Hardy-Sobolev-Maz'ya (HSM) inequalities. They derived a sharp Poincaré-Sobolev inequality \eqref{PoinSobIneq} in hyperbolic space , which has spurred further interest in analogous HSM inequalities \cite{VGM} involving higher-order derivatives (see \cite{Lu1, Lu2}). The authors in \cite{Li1} have thereafter explored the existence, non-existence, and symmetry of solutions for the higher-order Brezis-Nirenberg problem in hyperbolic space. This work relies on complex estimations of Green's functions for fractional Laplacians, Helgason-Fourier analysis, and Hardy-Littlewood-Sobolev inequality in the hyperbolic spaces. Specifically, the Hardy-Littlewood-Sobolev inequality in hyperbolic space can be stated as follows: For $0 < \lambda < N$ and $p = \frac{2N}{2N - \lambda}$, if $f, g \in L^p(\mathbb{B}^N)$,

\begin{equation*}
\left|\int_{\mathbb{B}^{N}} \int_{\mathbb{B}^{N}} \frac{f(x) g(y)}{\left(2 \sinh \frac{\rho(T_{y}(x))}{2}\right)^{\lambda}} \dvg(x) \dvg(y)\right| \leq C_{N, \lambda} \|f\|_{p} \|g\|_{p},
\end{equation*}
\noindent
where $T_{y}(x)$ denotes M\"obius transformation, $\rho$ denotes the hyperbolic distance (see Section \ref{prelim}), and

\begin{equation*}
C_{N, \lambda} = \pi^{\lambda / 2} \frac{\Gamma(N / 2 - \lambda / 2)}{\Gamma(N - \lambda / 2)} \left(\frac{\Gamma(N / 2)}{\Gamma(N)}\right)^{-1 + \lambda / N}
\end{equation*}
\noindent
is the best Hardy-Littlewood-Sobolev constant on $\mathbb{R}^{N}$. This constant is sharp, and there are no nonzero extremal functions. This result can be easily derived from the HLS inequality in $\mathbb{R}^N$ through a conformal change of metric to the Euclidean ball \cite{Lu1}. Furthermore, this inequality parallels the HLS inequality in $\Rn$ for the specific case when $p = r = \frac{2N}{N - \lambda}$, as stated in Theorem 4.3 of \cite{LL1}. The presence of  $\sinh \frac{\rho}{2}$ in the hyperbolic space context is attributed to the Green's function associated with the conformal Laplacian $-\Delta_{\mathbb{B}^N}-\frac{N(N-2)}{4}$ (see \cite{Lu2}). Furthermore, Theorem \ref{HlSIneqHybSp} states a generalized form of the HLS inequality for any complete Riemannian manifold, under an additional assumption regarding the heat kernel.

In light of the aforementioned HLS inequality, the author in \cite{Hh} examines an equation in the hyperbolic space that is analogous to the Choquard equation in $\mathbb{R}^N$. The study demonstrates the existence of a positive solution in the subcritical range, specifically for $\frac{N+\alpha}{N}<p<\frac{N+\alpha}{N -2}$ under the condition  $\lambda\leq \frac{(N-1)^2}{4}$, and also for $p=\frac{N+\alpha}{N-2}$ when $\frac{N(N-2)}{4}<\lambda\leq\frac{(N-1)^2}{4}$. Interestingly, the upper critical exponent $\frac{N+\alpha}{N -2}$ resembles the critical Sobolev exponent $\frac{2N}{N -2}$ that appears in the Brezis-Nirenberg problem, while the exponent $\frac{N + \alpha}{N}$ naturally emerges from the application of the HLS inequality.

Inspired by the aforementioned problems and the general Hardy-Littlewood-Sobolev inequality \eqref{HlSIneqHybSp} on a complete Riemannian manifold, we have investigated the problem \eqref{mainEq} in this article. The derivation of this inequality relies on Stein’s maximal ergodic theorem. A significant challenge while analzing problems like \eqref{mainEq} is the loss of compactness, even in the subcritical range, as the embedding $H^1(\mathbb{B}^N) \hookrightarrow L^{p}(\mathbb{B}^N)$ is non-compact for any $2 \leq p \leq \frac{2N}{N-2}$. This issue arises from the hyperbolic isometries in the infinite volume space $\bn$. In Euclidean spaces, such challenges are typically addressed by dilating a given sequence, but this approach is ineffective in $\mathbb{B}^N$ due to the equivalence of its conformal and isometry groups. To address this, we employ a concentration function approach near infinity. Another complexity arises from the non-local nature of the problem, driven by the green kernel of the fractional operator. In addition to these, several technical challenges emerge in this context. These include the nonlinear nature of hyperbolic translations (as described in Section \ref{HypTrans}) and the complexity of the semigroup formula for the inverse fractional Laplacian, which involves the heat kernel in hyperbolic space. Unlike the Euclidean case, where the heat kernel is given by $G_t(x)=\frac{1}{(4\pi t)^{N/2}}e^{-|x|^2/4t},\;x\in \Rn,\;t>0$, the heat kernel in hyperbolic space is far more complex (see Section \ref{SemiGp}). To the best of our knowledge, the problem described by \eqref{mainEq} has not been studied in the literature to date.

The problem \eqref{mainEq} can be analyzed using variational methods, thanks to the Hardy-Littlewood-Sobolev inequality \eqref{HLSconstant}. To establish the existence of finite energy solutions, it is essential to examine the energy functional associated with \eqref{hsm}, defined by:
$$
\mathcal{I}(u)=\frac{\int_{\mathbb{B}^{N}}\left[\left|\nabla_{\mathbb{B}^{N}} u(x)\right|^{2}-\lambda u^{2}(x)\right] \mathrm{~d}V_{\bn}}{\left(\int_{\mathbb{B}^{N}} \left[(- \Delta_{\mathbb{B}^{N}})^{-\frac{\alpha}{2}}|u|^p\right] |u|^{p} \dvg\right)^{\frac{1}{p}}},\;\; u\in \mathcal{H}_{\lambda}\left(\mathbb{B}^{N}\right) \backslash\{0\},$$
for $\frac{N+\alpha}{N}\leq p\leq \frac{N+\alpha}{N -2}$ and the space $\mathcal{H}_{\lambda}\left(\mathbb{B}^{N}\right)$ defined in Section \ref{prelim}. We establish the following main existence result for $p$ within the subcritical range: 
\begin{theorem}\label{exiSubCrit}
Let $0 < \alpha < N$ and $\frac{N + \alpha}{N} < p < 2_{\alpha}^{*} = \frac{N + \alpha}{N - 2}$ for $N \geq 3$. Then \eqref{mainEq} possesses a positive solution for any $\lambda \leq \frac{(N-1)^2}{4}$.
\end{theorem}
To prove the above theorem, we demonstrate the existence of a minimizer for the energy functional on the associated Nehari manifold. This proof hinges on a concentration-compactness argument at infinity and the fact that the equation under consideration is invariant under the isometry group in hyperbolic space. This minimizer is referred to as a \textit{ground state}. To establish the radial symmetry of the solutions, we utilize a symmetrization argument. This method is particularly effective in Choquard-type problems due to the presence of a Polya–Szego-like inequality, which implies a strong symmetrization effect on the nonlocal term. By using the minimality of the ground state, we derive a relationship between the function and its polarization, thereby establishing radial symmetry. Lastly, we obtain certain regularity results through classical bootstrap arguments.
\begin{remark}
   It is not trivial to establish a non-existence result analogous to \cite[Theorem 2]{MS1} in the critical case using the Pohožaev identity argument due to the nonlinear nature of the hyperbolic translation \eqref{HypTrans} involved in the definition of $\left(-\Delta_{\mathbb{B}^N}\right)^{-\alpha/2}$ (see \eqref{inverFracLap}). Nevertheless, by first deriving some a priori asymptotic estimates, one can prove the non-existence of solutions for $p = \frac{N + \alpha}{N - 2}$ and $\lambda \leq \frac{N(N-2)}{4}$.
\end{remark}
The paper is structured as follows: Section~\ref{prelim} introduces key notations and geometric definitions, along with preliminary concepts related to the hyperbolic space. Section~\ref{Existence} presents the proof of Theorem \ref{exiSubCrit}. In Section~\ref{Symmetry}, we state and prove the result concerning radial symmetry, while Section~\ref{regularity} addresses the regularity results. Numerous positive constants, whose specific values are irrelevant, are denoted by $C$.
\end{section}

\begin{section}{Preliminaries}\label{prelim}
In this section, we will introduce some of the notations and definitions used in this
paper and also recall some of the embeddings
related to the Sobolev space on the hyperbolic space. 
\smallskip
\noi
We will denote by $\bn$ the disc model of the hyperbolic space, i.e., the Euclidean unit ball $B(0,1):= \{x \in \mathbb{R}^N: |x|^2<1\}$ equipped with the Riemannian metric
\begin{align*}
	{\rm d}s^2 = \left(\frac{2}{1-|x|^2}\right)^2 \, {\rm d}x^2,
\end{align*}
where ${\rm d}x$ is the standard Euclidean metric and $|x|^2 = \sum_{i=1}^Nx_i^2$ is the standard Euclidean length. The corresponding volume element is given by $\mathrm{~d} V_{\mathbb{B}^{N}} = \big(\frac{2}{1-|x|^2}\big)^N {\rm d}x, $ where ${\rm d}x$ denotes the Lebesgue 
measure on $\rn$. $\nabla_{\bn}$ and $\Delta_{\bn}$ denote gradient 
 vector field and Laplace-Beltrami operator, respectively. Therefore in terms of local (global) coordinates $\nabla_{\bn}$ and $\Delta_{\bn}$ takes the form
\begin{align*} 
 \nabla_{\bn} = \left(\frac{1 - |x|^2}{2}\right)^2\nabla,  \quad 
 \Delta_{\bn} = \left(\frac{1 - |x|^2}{2}\right)^2 \Delta + (N - 2)\left(\frac{1 - |x|^2}{2}\right)  x \cdot \nabla,
\end{align*}
where $\nabla, \Delta$ are the standard Euclidean gradient vector field and Laplace operator, respectively, and '$\cdot$' denotes the 
standard inner product in $\mathbb{R}^N.$

 \smallskip 

\noindent

\subsection{M\"{o}bius Transformations and Convolution} \cite{LP}
\noindent
The M\"{o}bius transformations $T_a$ for each $a \in \bn$ are defined as follows: 
\begin{equation}
T_a(x)=\frac{|x-a|^2 a-\left(1-|a|^2\right)(x-a)}{1-2 x \cdot a+|x|^2|a|^2} \label{HypTrans}
\end{equation}
where $x \cdot a=x_1 a_1+x_2 a_2+\cdots+x_n a_n$ represents the scalar product in $\mathbb{R}^N$. The measure on $\bn$ is known to be invariant under M\"{o}bius transformations\\
The convolution of measurable functions $f$ and $g$ on $\bn$ can be defined using the Möbius transformations as follows:
\begin{equation*}
(f * g)(x)=\int_{\bn} f(y) g\left(T_x(y)\right)\dvg(y)
\end{equation*}
provided this integral exists. It is straightforward to verify that
\begin{equation*}
f * g=g * f.
\end{equation*}
\subsection{Hyperbolic Distance on $\bn$} The distance between $x$ and $y$ in $\bn$ can be defined as follows utilizing the M\"obius transformations:
\begin{equation*}
\rho(x, y)=\rho\left(T_x(y)\right)=\rho\left(T_y(x)\right)=\log \frac{1+\left|T_y(x)\right|}{1-\left|T_y(x)\right|}
\end{equation*}
where $\rho(x)= d(x,0)= \log \frac{1+|x|}{1-|x|}$ is the geodesic distance from the origin.\\
As a result, a subset of $\bn$ is a hyperbolic sphere in $\bn$ if and only if it is a Euclidean sphere in $\mathbb{R}^N$ and contained in $\bn$, possibly 
with a different centre and different radius, which can be explicitly computed from the formula of $d(x,y)$ \cite{RAT}.  Geodesic balls in $\bn$ of radius $r$ centred at $x \in \bn$ will be denoted by 
$$
B_r(x) : = \{ y \in \bn : d(x, y) < r \}.
$$
Additionally, if $g$ is radial, meaning $g=g(\rho)$, then for $f, g, h \in L^1\left(\bn\right)$, we have
\begin{equation*}
(f * g) * h=f *(g * h).
\end{equation*}
\subsection{ A sharp Poincar\'{e}-Sobolev inequality} (see \cite{MS})

\medskip

\noi{\bf Sobolev Space :} We will denote by ${H^{1}}(\bn)$ the Sobolev space on the disc
model of the hyperbolic space $\bn$, equipped with norm $\|u\|=\left(\int_{\mathbb{B}^N} |\nabla_{\mathbb{B}^{N}} u|^{2}\right)^{\frac{1}{2}},$
where  $|\nabla_{\bn} u| $ is given by
$|\nabla_{\bn} u| := \langle \nabla_{\bn} u, \nabla_{\bn} u \rangle^{\frac{1}{2}}_{\bn} .$ 

\noindent
For $N \geq 3$ and every $p \in \left(1, \frac{N+2}{N-2} \right]$ there exists an optimal constant 
$S_{\lambda,p} > 0$ such that
\begin{equation}
	S_{\lambda,p} \left( \int_{\mathbb{B}^{N}} |u|^{p + 1} \mathrm{~d} V_{\mathbb{B}^{N}} \right)^{\frac{2}{p + 1}} 
	\leq \int_{\mathbb{B}^N} \left[|\nabla_{\mathbb{B}^{N}} u|^{2}
	- \frac{(N-1)^2}{4} u^{2}\right] \, \mathrm{~d} V_{\mathbb{B}^{N}}, \label{PoinSobIneq}
\end{equation}
for every $u \in C^{\infty}_{0}(\mathbb{B}^{N}).$ If $ N = 2$, then any $p > 1$ is allowed.

\noi A crucial point to note is that the bottom of the spectrum of $- \Delta_{\bn}$ on $\bn$ is 
\begin{equation}\label{firsteigen}
	\frac{(N-1)^2}{4} = \inf_{u \in H^{1}(\bn)\setminus \{ 0 \}} 
	\dfrac{\int_{\bn}|\nabla_{\bn} u|^2 \, \mathrm{~d} V_{\mathbb{B}^{N}} }{\int_{\bn} |u|^2 \, \mathrm{~d} V_{\mathbb{B}^{N}}}. 
\end{equation}

\begin{remark}
	As a result of \eqref{firsteigen}, if $\lambda < \frac{(N-1)^2}{4},$ then 
\begin{equation}
 \left\|u\right\|_{\lambda} := \left[ \int_{\bn} \left( |\nabla_{\bn} u|^2 - \lambda \, u^2 \right) \, \mathrm{~d} V_{\mathbb{B}^{N}} \right]^{\frac{1}{2}}, \quad u \in C_0^{\infty}(\bn)\label{norm}
\end{equation}
	is a norm, equivalent to the $H^1(\bn)$ norm. When $\lambda = \frac{(N-1)^2}{4}$, the sharp Poincaré inequality \eqref{PoinSobIneq} ensures that $\|u\|_{\frac{(N-1)^2}{4}}$ is also a norm on $C_0^{\infty}\left(\mathbb{B}^N\right)$.
\end{remark}
For $\lambda \leq \frac{(N-1)^2}{4}$, denote by $\mathcal{H}_\lambda\left(\mathbb{B}^N\right)$ the completion of $C_0^{\infty}\left(\mathbb{B}^N\right)$ with respect to the norm $\|u\|_\lambda$. The associated inner product is denoted by $\langle \cdot, \cdot\rangle_{{\lambda}}.$ Note that
\begin{equation*}
S_{\lambda, p}\left(\int|u|^{p+1} \, d V_x\right)^{\frac{2}{p+1}} \leq \|u\|_\lambda, \quad p \in\left(1, \frac{N+2}{N-2}\right] \quad \text{for} \quad u \in \mathcal{H}_\lambda\left(\mathbb{B}^N\right).
\end{equation*}
Additionally, throughout this article, $\|\cdot\|_{r}$ denotes the $L^r$-norm with respect to the volume measure for $1 \leq r \leq \infty.$

\subsection{Helgason Fourier Transform on the Hyperbolic Space}
Now we will make sense of the green kernel of the fractional operator used in \eqref{mainEq}. To this aim we shall use the Helgason Fourier transform as follows (for further details, refer to \cite{BGS,LP} and the references therein).\\
Analogous to the Euclidean context, the Fourier transform can be defined as follows,
\begin{equation*}
\hat{f}(\beta, \theta)=\int_{\mathbb{B}^N} f(x) h_{\beta, \theta}(x) \dvg,
\end{equation*}
for $\beta \in \mathbb{R}, \theta \in \mathbb{S}^{N-1}$, where $h_{\beta, \theta}$ are the generalized eigenfunctions of the Laplace Beltrami operator that satisfy
\begin{equation*}
\Delta_{\mathbb{B}^N} h_{\beta, \theta}=-\left(\beta^2+\frac{(N-1)^2}{4}\right) h_{\beta, \theta} .
\end{equation*}
Furthermore, given $f \in C_0^\infty(\bn)$, the following inversion formula is valid:
\begin{equation*}
f(x)=\int_{-\infty}^{\infty} \int_{\mathbb{S}^{N-1}} \bar{h}_{\beta, \theta}(x) \hat{f}(\beta, \theta) \frac{\mathrm{~d} \theta \mathrm{~d}\beta}{|c(\beta)|^2},
\end{equation*}
where  $c(\beta)$ represents the Harish-Chandra coefficient:
\begin{equation*}
\frac{1}{|c(\beta)|^2}=\frac{1}{2} \frac{\left|\Gamma\left(\frac{N-1}{2}\right)\right|^2}{|\Gamma(N-1)|^2} \frac{\left\lvert\, \Gamma\left(i \beta+\left.\left(\frac{N-1}{2}\right)\right|^2\right.\right.}{|\Gamma(i \beta)|^2} . \\
\end{equation*}
Moreover, the following Plancherel formula holds
\begin{equation*}
\int_{\mathbb{B}^N}|f(x)|^2 \dvg=\int_{\mathbb{R} \times \mathbb{S}^{N-1}}|\hat{f}(\beta, \theta)|^2 \frac{\mathrm{~d} \theta \mathrm{~d}\beta}{|c(\beta)|^2} .
\end{equation*}

It is straightforward to verify through integration by parts for compactly supported functions, and thus for every $f \in L^2\left(\mathbb{B}^N\right)$ that
\begin{equation}
\begin{aligned}
\widehat{\Delta_{\mathbb{B}^N} f}(\beta, \theta) & =\int_{\mathbb{B}^N} \Delta_{\mathbb{B}^N} f(x) h_{\beta, \theta}(x) \dvg(x) =\int_{\mathbb{B}^N} f(x) \Delta_{\mathbb{B}^N} h_{\beta, \theta}(x) \dvg(x) \\
& =-\left(\beta^2+\frac{(N-1)^2}{4}\right) \hat{f}(\beta, \theta) .
\end{aligned} \label{lapHFT}
\end{equation}

Given the theory discussed above, the fractional Laplacian on the hyperbolic space can be defined as  ${\left(-\Delta_{\mathbb{B}^N}\right)}^{\gamma} f$, which is the operator satisfying
\begin{equation*}
(-\widehat{\left.\Delta_{\mathbb{B}^N}\right)^\gamma} f=\left(\beta^2+\frac{(N-1)^2}{4}\right)^{\gamma} \hat{f}, \quad \gamma \in \mathbb R. 
\end{equation*}

\medskip
However, to derive an explicit pointwise expression for the inverse Laplacian, we will utilize the following semigroup framework. To start, we introduce some notations related to the heat kernel in the hyperbolic space.
 \subsection{Semigroup Approach to the Fractional Laplacian}\label{SemiGp}
Consider the function $u=u(x, t)$, defined for $x \in \mathbb{B}^N$ and $t \geq 0$, which solves the heat equation over the entire space $\mathbb{B}^N$ with initial condition $f$ :
\begin{equation*}
\left\{\begin{array}{l}
\partial_t u(x, t)=\Delta_{\bn} u(x, t), \quad(x, t) \in \bn \times \mathbb{R}_{+}, \\
u(x, 0)=f(x) \in C_0^{\infty}\left(\bn\right),
\end{array}\right.
\end{equation*}
where, for convenience, $u(x, t)$ is considered to be $C^{\infty}$ and compactly supported in the spatial variable.

For every fixed $t$, when we apply the Fourier transform in the variable $x$ and use \eqref{lapHFT}, we obtain
\begin{equation*}
\left\{\begin{array}{l}
\partial_t \hat{u}(\beta, \theta, t)=-\left(\beta^2+\frac{(N-1)^2}{4}\right)\hat{u}(\beta, \theta, t), \;\text { for } t>0, \\
\hat{u}(\beta, \theta, 0)=\hat{f}(\beta, \theta),
\end{array}\right.
\end{equation*}
which yields
\begin{equation*}
\hat{u}(\beta, \theta, t)=e^{-t\left(\beta^2+\frac{(N-1)^2}{4}\right)} \hat{f}(\beta, \theta)= \widehat{e^{t \Delta_{\mathbb{B}^{N}}}f(\beta, \theta)}.
\end{equation*}
where $f \longmapsto e^{t \Delta_{\bn}} f$ denotes the solution operator. It is well-established (see \cite{LP}) that 
\begin{equation*}
u(x, t) \equiv e^{t \Delta_{\bn}} f(x)=p_{t,N} * f(x)= \int_{\bn} p_{t,N}(T_x(y))f(y) \dvg(y)
\end{equation*}
where $p_{t,N}(x,y)= p_{t,N}(\rho(x,y))$ is the heat kernel on $\bn$. 
The explicit formulas of heat kernel are given by (see \cite[Theorem 7.3]{LP}, \cite{GN}):\\
If $N=2 m+1$, then
\begin{equation*}
p_{t,2m+1}(\rho)=2^{-m-1} \pi^{-m-1 / 2} t^{-\frac{1}{2}} e^{-\frac{(N-1)^2}{4} t}\left(-\frac{1}{\sinh \rho} \frac{\partial}{\partial \rho}\right)^m e^{-\frac{\rho^2}{4 t}}. 
\end{equation*}
If $N=2 m$, then
\begin{equation*}
p_{t,2m}(\rho)=(2 \pi)^{-m-\frac{1}{2}} t^{-\frac{1}{2}} e^{-\frac{(N-1)^2}{4} t} \int_\rho^{+\infty} \frac{\sinh r}{\sqrt{\cosh r-\cosh \rho}}\left(-\frac{1}{\sinh r} \frac{\partial}{\partial r}\right)^m e^{-\frac{r^2}{4 t}} \mathrm{~d} r. 
\end{equation*}
Numerous researchers have investigated bounds for the heat kernel. For instance, Fabio Punzo \cite[Proposition 2.2]{punzo} derived the $L^{\infty}$ estimate, which helps in establishing the Hardy-Littlewood-Sobolev inequality in the hyperbolic space (see Theorem \ref{HLSinqHyb}). Additionally, the heat kernel $p_{t,N}$ satisfies the following bounds (see \cite{Davies}): for $N \geq 2$, there exist positive constants $A_N$ and $B_N$ such that
\begin{equation}
A_N h_N(t, x, y) \leq p_{t,N}(x,y) \leq B_N h_N(t, x, y) \quad \text{ for all } t>0 \text{ and } x, y \in \mathbb{B}^N, \label{HeatKerBounds}
\end{equation}
where $h_N(t, x, y)$ is defined as
\begin{equation*}
h_N(t, x, y):=h_N(t, r)=(4 \pi t)^{-\frac{N}{2}} e^{-\frac{(N-1)^2 t}{4}-\frac{(N-1) r}{2}-\frac{r^2}{4 t}}(1+r+t)^{\frac{(N-3)}{2}}(1+r),
\end{equation*}
with $r:=r(x, y)=\operatorname{dist}(x, y)$.\\
We can now establish the semigroup definition of the inverse fractional Laplacian, derived from the following numerical identity:
\begin{equation*}
\zeta^{-\gamma}=\frac{1}{\Gamma(\gamma)} \int_0^{\infty} e^{-t \zeta} \frac{\mathrm{~d}t}{t^{1-\gamma}} \quad \text { for any } \zeta>0, \gamma>0.
\end{equation*}
Setting $\zeta=\beta^2+\frac{(N-1)^2}{4}$ and multiplying the equation by $\hat{f}(\beta, \theta)$, we obtain
\begin{equation*}
\left(\beta^2+\frac{(N-1)^2}{4}\right)^{-\gamma}\hat{f}(\beta, \theta)=\frac{1}{\Gamma(\gamma)} \int_0^{\infty}\hat{f}(\beta, \theta) e^{-t \left(\beta^2+\frac{(N-1)^2}{4}\right)} \frac{\mathrm{~d}t}{t^{1-\gamma}}.
\end{equation*}
Taking the inverse transform then yields
\begin{equation*}
\begin{aligned}
  {\left(-\Delta_{\mathbb{B}^N}\right)}^{-\gamma}f(x)&= \frac{1}{\Gamma(\gamma)} \int_{-\infty}^{\infty} \int_{\mathbb{S}^{N-1}} \bar{h}_{\beta,\theta}(x) \left(\int_0^{\infty}\hat{f}(\beta, \theta) e^{-t \left(\beta^2+\frac{(N-1)^2}{4}\right)} \frac{\mathrm{~d}t}{t^{1-\gamma}}\right) \frac{d \theta d \beta}{|c(\beta)|^2}\\
   &=\frac{1}{\Gamma(\gamma)}\int_0^{\infty} \int_{-\infty}^{\infty} \int_{\mathbb{S}^{N-1}} \bar{h}_{\beta, \theta}(x)\hat{f}(\beta, \theta) e^{-t \left(\beta^2+\frac{(N-1)^2}{4}\right)}\frac{d \theta d \beta}{|c(\beta)|^2}\frac{\mathrm{~d}t}{t^{1-\gamma}}\\
   &= \frac{1}{\Gamma(\gamma)}\int_0^{\infty} \int_{-\infty}^{\infty} \int_{\mathbb{S}^{N-1}} \bar{h}_{\beta, \theta}(x)\widehat{e^{t \Delta_{\mathbb{B}^{N}}}f(\beta, \theta)}\frac{d \theta d \beta}{|c(\beta)|^2}\frac{\mathrm{~d}t}{t^{1-\gamma}}\\
   &= \frac{1}{\Gamma(\gamma)}\int_0^{\infty}e^{t \Delta_{\mathbb{B}^{N}}}f(x)\frac{\mathrm{~d}t}{t^{1-\gamma}}.
   \end{aligned}
\end{equation*}

Up to this point, we have developed the above heuristics for functions $f \in C_0^\infty(\mathbb{B}^N)$. Next, we aim to identify a class of functions for which the aforementioned integral converges. To this end, we define the function $f^{*}(x) = \sup_{t>0} |e^{t \Delta_{\mathbb{B}^N}} f(x)|$ for $f \in L^p(\mathbb{B}^N)$, where $1 < p < \infty$. Note that $f^*$ is a well-defined measurable function, as established in \cite{EStein}. We now present the following result regarding the convergence of the integral.

\begin{theorem}\label{InverseOpDom}
Suppose $N > 0$, $0 < \alpha < N$, and $1 < p < \frac{N}{\alpha}$, with $f \in L^p(\mathbb{B}^N)$. Then
\begin{equation}
{\left(-\Delta_{\mathbb{B}^N}\right)}^{-\alpha/2}f(x):= \frac{1}{\Gamma(\alpha/2)}\int_0^{\infty}e^{t \Delta_{\mathbb{B}^{N}}}f(x)\frac{d t}{t^{1-\alpha/2}} \label{semiGpInvDefn}
\end{equation}
is convergent.
\end{theorem}

Before proving this theorem, we establish the following on-diagonal and off-diagonal estimates for the heat kernel on the hyperbolic space.\\
Using \eqref{HeatKerBounds}, we have 
\begin{equation}
\begin{aligned}
    p_{t,N}(x,x) \leq B_N (4 \pi t)^{-\frac{N}{2}} e^{-\frac{(N-1)^2 t}{4}}(1+t)^{\frac{(N-3)}{2}} &\leq B_N (4 \pi t)^{-\frac{N}{2}} e^{-\frac{(N-1)^2 t}{4}} e^{\frac{t(N-3)}{2}}\\ 
    & \leq B_N (4 \pi t)^{-\frac{N}{2}} e^{\frac{-t\left((N-2)^2+3\right)}{4}}\\
    & \leq C\;t^{-\frac{N}{2}}\;\;\;\forall t>0.
\end{aligned}\label{DiagEst}
\end{equation}
Further, recall the following semigroup property of the heat kernel:
\begin{equation}
p_{t+s,N}(x, y)=\int_{\bn} p_{t,N}(x, z) p_{s,N}(z, y) \dvg(z) \label{semiGpPro}
\end{equation}
for $x, y \in \bn$ and $t, s>0$.\\
Specifically, by setting $y=x$ and $s=t$ in \eqref{semiGpPro} and noting the radial nature of the heat kernel, we obtain
\begin{equation}
p_{2t,N}(x, x)=\int_{\bn} \left(p_{t,N}(x, z)\right)^2 \dvg(z). \label{equivProp}
\end{equation}
By combining \eqref{DiagEst} and \eqref{equivProp}, we can deduce that $p_{t,N}(x,y) \in L^2(\bn)$ as a function of $y$.\\
Thus, using \eqref{semiGpPro}, \eqref{equivProp} and \eqref{DiagEst}, we have for $t>0$
\begin{equation}
    \begin{aligned}
      p_{t,N}(x,y) &=  \int_{\bn} p_{\frac{t}{2},N}(x, z) p_{\frac{t}{2},N}(z, y) \dvg(z)\\
      &\leq \left(\int_{\bn}\left( p_{\frac{t}{2},N}(x, z)\dvg(z)\right)^2 \right)^\frac{1}{2}\left(\int_{\bn}\left( p_{\frac{t}{2},N}(y, z)\dvg(z)\right)^2 \right)^\frac{1}{2}\\
      & = \left(p_{t,N}(x,x)\right)^{\frac{1}{2}}\left(p_{t,N}(y,y)\right)^{\frac{1}{2}} \leq C t^{-\frac{N}{2}} \quad\forall x,y \in \bn. \label{offDiagEst}
    \end{aligned}
\end{equation}
We now present the necessary tools to establish Theorem \ref{InverseOpDom}.
\begin{proof}[Proof of Theorem \ref{InverseOpDom}]
Fix $x \in \bn$, and consider the integral in \eqref{semiGpInvDefn}
\begin{equation*}
 \frac{1}{\Gamma(\alpha/2)}\int_0^{\infty}e^{t \Delta_{\mathbb{B}^{N}}}f(x)\frac{\mathrm{~d}t}{t^{1-\alpha/2}}= \underbrace{\int_{0}^{1}}_{P} + \underbrace{\int_{1}^{\infty}}_{Q}.
\end{equation*}
Due to the estimate \eqref{DiagEst}, we can apply Proposition 2.2 from \cite{punzo} to obtain
\begin{equation*}
   |e^{t \Delta_{\bn}}f(x)| \leq \frac{C^{1/p}}{t^{N/2p}}\|f\|_p.
\end{equation*}
Using this estimate, we have
\begin{equation*}
\left|Q\right| \leq \frac{C^{1 / p}}{\Gamma(\alpha / 2)} \|f\|_p\int_1^{+\infty} t^{\frac{\alpha}{2}-\frac{N}{2 p}-1} \mathrm{~d}t =\frac{2C^{1 / p}\|f\|_p}{(N-\alpha/p) \Gamma(\alpha / 2)}<\infty.
\end{equation*}
Finally,
\begin{equation*}
\begin{aligned}
\left|P\right| &= \left|\frac{1}{\Gamma(\alpha/2)}\int_0^{1}e^{t \Delta_{\mathbb{B}^{N}}}f(x)\frac{\mathrm{~d}t}{t^{1-\alpha/2}}\right| \leq \frac{1}{\Gamma(\alpha / 2)} \int_0^1 t^{\alpha / 2-1} f^*(x) \mathrm{~d}t=\frac{2}{\alpha} \frac{1}{\Gamma(\alpha / 2)} f^*(x)<\infty,
\end{aligned}
\end{equation*}
which completes the proof.
\end{proof}

\medskip
\noindent
For simplicity, we define
\begin{equation}
k_{\alpha,N}(\rho)= \frac{1}{\Gamma(\alpha/2)} \int_{0}^{\infty} p_{t,N}(\rho)\frac{\mathrm{~d}t}{t^{1-\alpha/2}} \label{GreenFunc}.
\end{equation}
Then,
\begin{equation}
  {\left(-\Delta_{\mathbb{B}^N}\right)}^{-\alpha/2}f(x)= k_{\alpha,N}*f(x). \label{inverFracLap}
\end{equation}

\medskip
To establish the finiteness of the energy functional associated with \eqref{mainEq}, we will rely on the  following Hardy-Sobolev-Littlewood inequality in the hyperbolic space. This connection between HLS theory and heat kernel estimates is attributed to Varopoulos.
\begin{theorem}\label{HLSinqHyb}[Hardy-Littlewood-Sobolev inequality]
Let $(\mathbb{M}, g)$ be a complete Riemannian manifold with dimension $N \ge 1$, and let $L$ denote the Laplace-Beltrami operator on $\mathbb{M}$, $0 < \alpha < N$ and $1 < s < \frac{N}{\alpha}$. If there exists a constant $C > 0$ such that for all $t > 0$ and for every $x, y \in \mathbb{M}$,
\begin{equation*}
p(x, y, t) \leq \frac{C}{t^{N / 2}},
\end{equation*}
where $p(x, y, t)$ represents the heat kernel associated to the semigroup corresponding to $L$. Then, for every function $f \in L_\mu^s(\mathbb{M})$, the following inequality holds
\begin{equation}
\left\|(-L)^{-\alpha / 2} f\right\|_{\frac{N s}{N - s \alpha}} \leq \widetilde{C}(N,\alpha, s) \|f\|_s, \label{HlSIneqHybSp}
\end{equation}
where
\begin{equation}
\widetilde{C}(N,\alpha, s) = \left(\frac{s}{s - 1}\right)^{1 - \alpha / N} \frac{2 N C^{\alpha / N}}{\alpha (N - s \alpha) \Gamma(\alpha / 2)}. \label{HLSconstant}
\end{equation}
\end{theorem}
\noindent
For a detailed proof of this theorem, see \cite{FBa}.
\end{section}
\medskip
\begin{section}{Existence}\label{Existence}
To establish the proof of Theorem \ref{exiSubCrit}, we recall the energy functional associated with \eqref{mainEq} is defined as
$$
\mathcal{I}(u)=\frac{\int_{\mathbb{B}^{N}}\left|\nabla_{\mathbb{B}^{N}} u(x)\right|^{2}-\lambda u^{2}(x) \mathrm{~d}V_{\bn}}{\left(\int_{\mathbb{B}^{N}} \left[(- \Delta_{\mathbb{B}^{N}})^{-\frac{\alpha}{2}}|u|^p\right] |u|^{p} \dvg\right)^{\frac{1}{p}}},\;\; u\in \mathcal{H}_{\lambda}\left(\mathbb{B}^{N}\right) \backslash\{0\},
$$
\noindent
and
\noindent
$$
\zeta=\inf _{u \in \mathcal{H}_{\lambda}\left(\mathbb{B}^{N}\right)\setminus \{0\}} \mathcal{I}(u) .
$$
Observe that, owing to Theorem \ref{HLSinqHyb}, the functional $\mathcal{I}(u)$ is well-defined. For $u \in \mathcal{H}_{\lambda}(\bn)$, it suffices to ensure the finiteness of the integral in the denominator of $\mathcal{I}(u)$. By applying Hölder's inequality and Theorem \ref{HLSinqHyb} with $s = \frac{2N}{N + \alpha}$, we obtain
\begin{equation*}
\begin{aligned}
   & \int_{\mathbb{B}^{N}} \left[(- \Delta_{\mathbb{B}^{N}})^{-\frac{\alpha}{2}}|u|^p\right] |u|^{p} \dvg\\
    & \leq \left(\int_{\bn}\left|(- \Delta_{\mathbb{B}^{N}})^{-\frac{\alpha}{2}}|u|^p\right|^{\frac{2N}{N-\alpha}}\dvg\right)^{\frac{N-\alpha}{2N}} \left(\int_{\bn}|u|^{\frac{2Np}{N+\alpha}}\dvg\right)^{\frac{N+\alpha}{2N}}\\
    & \leq C(N,\alpha)\int_{\bn}\left(|u|^{\frac{2Np}{N+\alpha}}\dvg\right)^{\frac{N+\alpha}{N}},
    \end{aligned} 
\end{equation*}
where $C(N, \alpha) = \widetilde{C}(N, \alpha, \frac{2N}{N + \alpha})$ is defined in \eqref{HLSconstant}.  Therefore, the integral
$$
\int_{\mathbb{B}^{N}} \left[(- \Delta_{\mathbb{B}^{N}})^{-\frac{\alpha}{2}}|u|^p\right] |u|^{p} \dvg
$$
is well-defined as $|u|^p \in L^{\frac{2N}{N+\alpha}}(\mathbb{B}^N)$. This is guaranteed by the assumptions $u \in \mathcal{H}_{\lambda}(\mathbb{B}^N)$, the range of $p$, i.e., $\frac{N + \alpha}{N} \leq p \leq 2_{\alpha}^{*} = \frac{N + \alpha}{N - 2}$, and the sharp Poincaré-Sobolev inequality.

Observe that \eqref{mainEq} is invariant under the isometry group of $\mathbb{B}^N$, which is identical to its conformal group. To demonstrate the existence result using a concentration-compactness type argument, define for $r > 0$,
$$
S_{r}=\left\{x \in \mathbb{R}^{N}:|x|^{2}=1+r^{2}\right\},
$$
\noindent
and for $a \in S_{r}$, define
\noindent
$$
A(a, r)=B(a, r) \cap \mathbb{B}^{N}
$$
\noindent
where $B(a, r)$ is the open ball in the Euclidean space with center $a$ and radius $r>0$. Moreover, for the choice of $a$ and $r$, $\partial B(a, r)$ is orthogonal to $\mathbb{S}^{N-1}$.
Let's recall a lemma from \cite{BS} before proving the theorem:

\begin{lemma} \label{isomeLemm}
Let $r_{1} > 0$ and $r_{2} > 0$, and let $A(a_{i}, r_{i})$ for $i = 1, 2$ be as defined above. Then there exists $\tau \in I(\mathbb{B}^{N})$ such that $\tau(A(a_{1}, r_{1})) = A(a_{2}, r_{2})$, where $I(\mathbb{B}^{N})$ is the isometry group of $\mathbb{B}^{N}$.
\end{lemma}
\noindent

\medskip
Before proceeding to the main result, we establish the following variant of the Brezis-Lieb lemma (\cite{BL1}):
\begin{lemma}
Consider $0 < \alpha < N$ and $\frac{N + \alpha}{N}<p<\frac{N + \alpha}{N - 2}$. Suppose $\{u_n\}$ is a bounded sequence in $L^{\frac{2Np}{N + \alpha}}(\mathbb{B}^N)$ such that $u_n \to u$ almost everywhere in $\mathbb{B}^N$ as $n \to \infty$. Then, we have
$$
\begin{aligned}
& \lim _{n \rightarrow \infty} \int_{\mathbb{B}^{N}} \left((- \Delta_{\mathbb{B}^{N}})^{-\frac{\alpha}{2}}|u_n|^p\right)  |u_n|^{p} \dvg -\int_{\mathbb{B}^{N}} \left((- \Delta_{\mathbb{B}^{N}})^{-\frac{\alpha}{2}}|u_n-u|^p\right) |u_n-u|^{p} \dvg\\
& \hspace{1cm}=\int_{\mathbb{B}^{N}} \left((- \Delta_{\mathbb{B}^{N}})^{-\frac{\alpha}{2}}|u|^p\right) |u|^{p} \dvg.
\end{aligned}
$$
\end{lemma}
\begin{proof}
First, observe that
\begin{equation}
\begin{aligned}
& \int_{\mathbb{B}^{N}} \left((- \Delta_{\mathbb{B}^{N}})^{-\frac{\alpha}{2}}|u_n|^p\right)  |u_n|^{p} \dvg - \int_{\mathbb{B}^{N}} \left((- \Delta_{\mathbb{B}^{N}})^{-\frac{\alpha}{2}}|u_n-u|^p\right) |u_n-u|^{p} \dvg \\
& = \int_{\mathbb{B}^{N}} \left((- \Delta_{\mathbb{B}^{N}})^{-\frac{\alpha}{2}}\left(|u_n|^p-|u_n-u|^p\right)\right) \left(|u_n|^p-|u_n-u|^p\right) \dvg\\
& \quad+2 \int_{\mathbb{B}^{N}} \left((- \Delta_{\mathbb{B}^{N}})^{-\frac{\alpha}{2}}\left(|u_n|^p-|u_n-u|^p\right)\right)|u_n-u|^p\dvg.
\end{aligned}\label{breakTerm}
\end{equation}

Next, since $u_{n} \in L^{\frac{2Np}{N + \alpha}}(\mathbb{B}^{N})$, by the Brezis-Lieb lemma, we have

\begin{equation}
\left|u_{n} - u\right|^p - \left|u_{n}\right|^p \rightarrow |u|^p \quad \text{in } L^{\frac{2N}{N + \alpha}}(\mathbb{B}^{N}).
\label{eq2}
\end{equation}
\noindent
By employing the Hardy-Littlewood-Sobolev inequality \eqref{HlSIneqHybSp} along with \eqref{eq2}, we derive the following
\begin{equation}
\begin{aligned}
&\int_{\mathbb{B}^{N}} \left((- \Delta_{\mathbb{B}^{N}})^{-\frac{\alpha}{2}}\left(|u_n|^p-|u_n-u|^p\right)\right) \left(|u_n|^p-|u_n-u|^p\right) \dvg\\
&\hspace{3cm}-\int_{\mathbb{B}^{N}} \left((- \Delta_{\mathbb{B}^{N}})^{-\frac{\alpha}{2}}|u|^p\right)  |u|^{p} \dvg  \\
&= \int_{\mathbb{B}^{N}} \left((- \Delta_{\mathbb{B}^{N}})^{-\frac{\alpha}{2}}\left(|u_n|^p-|u_n-u|^p-|u|^p\right)\right) \left(|u_n|^p-|u_n-u|^p\right) \dvg\\
&\hspace{3cm}+\int_{\mathbb{B}^{N}} \left((- \Delta_{\mathbb{B}^{N}})^{-\frac{\alpha}{2}}\left(|u|^p\right)\right) \left(|u_n|^p-|u_n-u|^p-|u|^p\right) \dvg\\
&\rightarrow 0.
\end{aligned}\label{int1}
\end{equation}
\noindent
Furthermore, for $u_n \in L^{\frac{2Np}{N + \alpha}}(\mathbb{B}^{N})$ with $u_n(x) \rightarrow u(x)$ a.e. in $\bn$, we have
\begin{equation}
\left|u_{n}-u\right|^{p} \rightharpoonup 0 \quad \text { weakly in } L^{\frac{2 N}{N+\alpha}}\left(\mathbb{B}^{N}\right). \label{eq3}
\end{equation}
Consider now the following integral
\begin{equation}
    \begin{aligned}
        &\int_{\mathbb{B}^{N}} \left((- \Delta_{\mathbb{B}^{N}})^{-\frac{\alpha}{2}}\left(|u_n|^p-|u_n-u|^p\right)\right)|u_n-u|^p \dvg\\
        &= \int_{\mathbb{B}^{N}} \left((- \Delta_{\mathbb{B}^{N}})^{-\frac{\alpha}{2}}\left(|u_n|^p-|u_n-u|^p-|u|^p\right)\right) |u_n-u|^p \dvg\\
        &\hspace{1cm}+\int_{\mathbb{B}^{N}} \left((- \Delta_{\mathbb{B}^{N}})^{-\frac{\alpha}{2}}|u|^p\right) |u_n-u|^p \dvg\\
        & \rightarrow 0,
    \end{aligned}\label{int2}
\end{equation}
which follows from \eqref{HlSIneqHybSp}, \eqref{eq2}, and \eqref{eq3}.\\
By combining \eqref{breakTerm}, \eqref{int1}, and \eqref{int2}, we arrive at the desired conclusion.
\end{proof}

\begin{lemma}
Let $0<\alpha<N$ and $\frac{N+\alpha}{N}<p<2_{\alpha}^{*}=\frac{N+\alpha}{N-2}$ if $N \geq 3$. Then for $\lambda \leq \frac{(N-1)^{2}}{4}, \zeta$ is attained by some nonnegative function in $\mathcal{H}_{\lambda}\left(\mathbb{B}^{N}\right)$. 
\end{lemma}

\begin{proof}
The associated Nehari manifold is defined as 
$$
\mathcal{N}:=\left\{u \in \mathcal{H}_{\lambda}\left(\mathbb{B}^{N}\right):\|u\|_{\lambda}^{2}=\int_{\mathbb{B}^{N}} \left((- \Delta_{\mathbb{B}^{N}})^{-\frac{\alpha}{2}}|u|^p\right) |u|^{p} \dvg, u \not\equiv 0\right\}.$$
\noindent
It is straightforward to see that
\begin{equation*}
\zeta=\inf_{u \in \mathcal{H}_{\lambda}\left(\mathbb{B}^{N}\right)\setminus \{0\}} \mathcal{I}(u)=\inf_{u \in \mathcal{N}} \mathcal{I}(u). 
\end{equation*}
\noindent
By utilizing the Hardy-Littlewood-Sobolev inequality \eqref{HlSIneqHybSp} and the Poincaré-Sobolev inequality \eqref{PoinSobIneq}, we obtain
\begin{equation}
\zeta \geq \left(C(N, \alpha)\right)^{-1/p}\;S_{\lambda, \frac{2 Np}{N+\alpha}}>0.\label{MinPos}
\end{equation}
We assert that $\zeta$ is attained.\\
Let $\left\{u_{n}\right\}$ be a minimizing sequence for $\zeta$ in $\mathcal{N}$, meaning
\begin{equation}
\left\|u_{n}\right\|_{\lambda}^{\frac{2(p-1)}{p}}=\left(\int_{\mathbb{B}^{N}} \left((- \Delta_{\mathbb{B}^{N}})^{-\frac{\alpha}{2}}|u_n|^p\right) |u_n|^{p} \dvg\right)^{\frac{p-1}{p}} \rightarrow \zeta \quad \text { as } n \rightarrow \infty \text {. } \label{MinSeqNehari}
\end{equation}

It is evident that $\left\{u_{n}\right\}$ is bounded in $\mathcal{H}_{\lambda}\left(\mathbb{B}^{N}\right)$ by some constant $M>0$. Consequently, we can assume, up to a subsequence, that $u_{n} \rightharpoonup u$ weakly in $H_{\lambda}\left(\mathbb{B}^{N}\right)$. To demonstrate that $\zeta$ is attained, it suffices to show that $\left\{u_{n}\right\}$ converges weakly, and pointwise, to some $u \in \mathcal{N}$ up to the isometry group of $\mathbb{B}^{N}$.

\medskip
The proof will be divided into the following steps:

\medskip
\textbf{Step 1:} The goal of this step is to identify a sequence of translations to restore compactness. Since $\left\{u_{n}\right\}$ is a minimizing sequence for $\zeta$, and in light of \eqref{MinPos} and \eqref{MinSeqNehari}, ${u_n}$ does not strongly converge to zero. Therefore, we have
$$
\liminf _{n \rightarrow \infty} \int_{\mathbb{B}^{N}} \left((- \Delta_{\mathbb{B}^{N}})^{-\frac{\alpha}{2}}|u_n|^p\right)  |u_n(x)|^{p} \dvg>\delta_{1}>0 .
$$
\noindent
By applying the Hardy-Littlewood-Sobolev inequality \eqref{HLSinqHyb}, it follows that
$$
\liminf _{n \rightarrow \infty} \int_{\mathbb{B}^{N}}\left|u_{n}\right|^{\frac{2 Np}{N+\alpha}} \dvg>\delta_{2}>0.
$$
Choose a positive constant $\delta$ such that
$$
0<2 \delta<\delta_{2}<\left((C(N, \alpha))^{-1} S_{\lambda, \frac{2 Np}{N+\alpha}}^{1+\frac{p}{2}} M^{-p}\right)^{\frac{2}{p-2} \frac{Np}{N+\alpha}} .
$$
\noindent
Define the concentration function $Q_n: (0, +\infty) \to \mathbb{R}$ by
\noindent
$$
\begin{aligned}
& Q_{n}(r)=\sup _{x \in S_{r}} \int_{A(x, r)}\left|u_{n}\right|^{\frac{2Np}{N+\alpha}} \dvg.
\end{aligned}
$$

\noindent
It follows that $\lim_{r \to 0} Q_n(r) = 0$ and $\lim_{r \to \infty} Q_n(r) > \delta$. Therefore, we can find $R_n > 0$ and $x_n \in S_{R_n}$ such that
\noindent
$$
\sup _{x \in S_{R_{n}}} \int_{A\left(x, R_{n}\right)}\left|u_{n}\right|^{\frac{2Np}{N+\alpha}} \dvg=\int_{A\left(x_n, R_{n}\right)}\left|u_{n}\right|^{\frac{2Np}{N+\alpha}} \dvg=\delta.
$$

\noindent
Fix  $x_0\in S_{\sqrt{3}}$, use Lemma \ref{isomeLemm} to select isometry $T_n \in I(\mathbb{B}^N)$ such that $A(x_n,r_n)=T_nA(x_0,\sqrt{3})$ for all $n$.

\medskip
\textbf{Step 2:} This step aims to locate the candidate for the minimizer.\\
Define $v_{n}(x)=u_{n} \circ T_{n}(x)$. With $T_{n}$ being an isometry, it is trivial to observe that $\left\{v_{n}(x)\right\} \subset \mathcal{N}$ and remains a minimizing sequence, i.e,
\begin{equation*}
\left\|v_{n}\right\|_{\lambda}^{\frac{2(p-1)}{p}}=\left(\int_{\mathbb{B}^{N}} \left((- \Delta_{\mathbb{B}^{N}})^{-\frac{\alpha}{2}}|v_n|^p\right) |v_n|^{p} \dvg\right)^{\frac{p-1}{p}}\rightarrow \zeta \quad \text { as } n \rightarrow \infty.
\end{equation*}

\noindent
Furthermore,
\begin{equation}
\begin{aligned}
\int_{A\left(x_{0}, \sqrt{3}\right)}\left|v_{n}\right|^{\frac{2Np}{N+\alpha}} \dvg & =\int_{A\left(x_{n}, R_{n}\right)}\left|u_{n}\right|^{\frac{2Np}{N+\alpha}} \dvg=\delta \\
& =\sup _{x \in S_{\sqrt{3}}} \int_{A(x, \sqrt{3})}\left|v_{n}\right|^{\frac{2Np}{N+\alpha}} \dvg. 
\end{aligned}\label{ConcOfNewSeq}
\end{equation}

\noindent
Following the Ekeland principle, we can assume that the sequence $\{v_{n}\}$ is a Palais-Smale sequence. Specifically,
\begin{equation}
\left\langle v_{n}, w\right\rangle_{\lambda}=\int_{\mathbb{B}^{N}} \left((- \Delta_{\mathbb{B}^{N}})^{-\frac{\alpha}{2}}|v_n|^p\right) |v_n(x)|^{p-2}v_n(x) w\dvg+o_n(1)\|w\| \text{ for $w\in\mathcal{H_{\lambda}}(\bn)$}, \label{TranPSseq}
\end{equation}
\noindent
where $o_n(1) \rightarrow 0$ as $n\rightarrow \infty$ in $\mathcal{H}_{\lambda}$. Consequently, up to a subsequence, we may assume $v_{n} \rightharpoonup v$ in $\mathcal{H}_{\lambda}\left(\mathbb{B}^{N}\right)$. Furthermore, due to this weak convergence and the continuity of the functional $u \mapsto J^ {\prime}(u) $ where $J(u)= \int_{\mathbb{B}^{N}} \left((- \Delta_{\mathbb{B}^{N}})^{-\frac{\alpha}{2}}|u|^p\right) |u|^{p} \dvg $, it is not difficult to deduce that
$$
\|v\|_{\lambda}^{2}=\int_{\mathbb{B}^{N}} \left((- \Delta_{\mathbb{B}^{N}})^{-\frac{\alpha}{2}}|v|^p\right) |v|^{p} \dvg.
$$
\noindent
Thus, it remains to show that $v \not\equiv 0$ to conclude that $v\in \mathcal{N}$.

\medskip
\textbf{Step 3:} Assume, if possible, that $v\equiv0$. We propose that for any $1>r>2-\sqrt{3}$,
$$
\int_{\mathbb{B}^{N} \cap\{|x| \geq r\}}\left|v_{n}\right|^{\frac{2Np}{N+\alpha}} \dvg=o(1) .
$$
To establish this, select a point $a \in S_{\sqrt{3}}$. Define $\Phi \in C_{0}^{\infty}(A(a, \sqrt{3}))$ such that $0 \leq \Phi \leq 1$, and substitute $w=\Phi^{2} v_{n}$ in \eqref{TranPSseq}. This yields
\begin{equation*}
\begin{aligned}
& \int_{\mathbb{B}^{N}} \nabla_{\mathbb{B}^{N}} v_{n} \nabla_{\mathbb{B}^{N}}\left(\Phi^{2} v_{n}\right)-\lambda v_{n} \Phi^{2} v_{n} \dvg =\int_{\mathbb{B}^{N}} \left((- \Delta_{\mathbb{B}^{N}})^{-\frac{\alpha}{2}}|v_n|^p\right) |v_n|^{p-2}v_n \Phi^{2} v_{n}\dvg+o(1) .
\end{aligned}
\end{equation*}

\medskip
\noindent
Using local compact embedding, we simplify the left-hand side of the above expression as in (\cite[pg. 255]{BS})
\begin{equation*}
\langle  v_n,\phi^2v_n\rangle_\lambda =   \norm{\phi v_n}_\la^2+o(1). 
\end{equation*}

\medskip
\noindent
Thus, combining the two expressions above gives 
\begin{equation}
\begin{aligned}
& \int_{\mathbb{B}^{N}}\left|\nabla_{\mathbb{B}^{N}}\left(\Phi v_{n}\right)\right|^{2}-\lambda\left(\Phi v_{n}\right)^{2} \dvg =\int_{\mathbb{B}^{N}} \left((- \Delta_{\mathbb{B}^{N}})^{-\frac{\alpha}{2}}|v_n|^p\right) |v_n|^{p-2}\left(\Phi v_{n}\right)^{2}\dvg+o(1).
\end{aligned}\label{eq1}
\end{equation}

\medskip
\noindent
Using \eqref{eq1}, along with the Hölder's inequality, the Hardy-Littlewood-Sobolev inequality \eqref{HLSinqHyb}, and the Poincaré-Sobolev inequality, we derive

$$
\begin{aligned}
& S_{\lambda,\frac{2Np}{N+\alpha}}\left(\int_{\mathbb{B}^{N}}\left|\Phi v_{n}\right|^{\frac{2Np}{N+\alpha}} \dvg\right)^{\frac{2}{\frac{2Np}{N+\alpha}}} \\&
\leq C(N, \alpha)\left(\int_{\mathbb{B}^{N}}\left|v_{n}\right|^{\frac{2Np}{N+\alpha}} \dvg\right)^{\frac{N+\alpha}{2 N}} \left(\int_{\mathbb{B}^{N}}\left[\left|v_{n}\right|^{p-2}\left(v_{n} \Phi\right)^{2}\right]^{\frac{2 N}{N+\alpha}} \dvg\right)^{\frac{N+\alpha}{2 N}} \\
& \leq C(N, \alpha) S_{\lambda,\frac{2 Np}{N+\alpha}}^{-\frac{p}{2}}\left\|u_{n}\right\|_{\lambda}^{p} \left(\int_{\mathbb{B}^{N}}\left[\left|v_{n}\right|^{p-2}\left(v_{n} \Phi\right)^{2}\right]^{\frac{2 N}{N+\alpha}} \dvg\right)^{\frac{N+\alpha}{2 N}} \\
& \leq C(N, \alpha) S_{\lambda, \frac{2 Np}{N+\alpha}}^{-\frac{p}{2}} M^{p} \left(\int_{\mathbb{B}^{N}}\left|\Phi v_{n}\right|^{\frac{2 Np}{N+\alpha}} \dvg\right)^{\frac{2}{p} \frac{N+\alpha}{2 N}} \left(\int_{A(a, \sqrt{3})}\left|v_{n}\right|^{\frac{2Np}{N+\alpha}} \dvg\right)^{\frac{p-2}{p} \frac{N+\alpha}{2 N}} .
\end{aligned}
$$
\noindent
Now, if
\noindent
$$
\int_{\mathbb{B}^{N}}\left|\Phi v_{n}(x)\right|^{\frac{2Np}{N+\alpha}} \dvg(x) \nrightarrow 0,
$$
\noindent
we have

$$
\begin{aligned}
\delta^{\frac{p-2}{p} \frac{N+\alpha}{2 N}} & >\left(\int_{A(a, \sqrt{3})}\left|v_{n}\right|^{\frac{2Np}{N+\alpha}} \dvg\right)^{\frac{p-2}{p} \frac{N+\alpha}{2 N}} \\
& \geq(C(N, \alpha))^{-1} S_{\lambda, \frac{2 Np}{N+\alpha}}^{1+\frac{p}{2}} M^{-p}>\delta^{\frac{p-2}{p} \frac{N+\alpha}{2 N}},
\end{aligned}
$$

\noindent
which is a contradiction. Therefore,

$$
\int_{\mathbb{B}^{N}}\left|\Phi v_{n}(x)\right|^{\frac{2Np}{N+\alpha}} \dvg(x) \rightarrow 0.
$$

\noindent
Since $a \in S_{\sqrt{3}}$ is arbitrary, the claim is proven.\\
\noindent
Since the condition $\frac{N+\alpha}{N} < p < \frac{N+\alpha}{N-2}$ ensures that $v_{n} \rightarrow 0$ in $L_{\text{loc}}^{\frac{2Np}{N+\alpha}}(\mathbb{B}^{N})$, this, combined with the previous assertion, directly contradicts \eqref{ConcOfNewSeq}. Consequently, we conclude that $v \not\equiv 0$ and $v \in \mathcal{N}$.
\end{proof}
Thus, the proof of Theorem \ref{exiSubCrit} is straightforwardly derived from the preceding theorem and the application of the maximum principle.
\end{section}

\begin{section}{Radial Symmetry}\label{Symmetry}
The objective of this section is to demonstrate the hyperbolic symmetry of ground states. Specifically, we will prove the following result.
\begin{proposition}\label{RadSymm}
  Let $0 < \alpha < N$ and $\frac{N+\alpha}{N} \leq p \leq 2_{\alpha}^{*} = \frac{N+\alpha}{N-2}$ for $N \geq 3$. If $u \in \mathcal{H}_{\lambda}(\mathbb{B}^{N})$ is a positive ground state of \eqref{mainEq}, then there exists a point $x_0 \in \mathbb{B}^N$ and a function $v:(0, \infty) \rightarrow \mathbb{R}$, which is non-negative and non-increasing, such that for almost every $x \in \mathbb{B}^N$
\begin{equation*}
u(x)=v\left(d(x,x_0)\right) .
\end{equation*}
\end{proposition}
\medskip
To demonstrate the radial symmetry, we will utilize the polarization technique, which is also referred to as the method of symmetrization or rearrangement. To facilitate this, we first introduce some notations for use throughout this section. Fix an origin $e$ in $\mathbb{B}^N$. Let $H \subset \mathbb{B}^N$ be a closed, totally geodesic hypersurface such that $e \notin H$. The components of $\mathbb{B}^N \setminus H$ are denoted as $H^+$ and $H^-$, where $H^+$ contains $e$ and $H^-$ does not. Let $\sigma_H: \mathbb{B}^N \to \mathbb{B}^N$ represent the reflection w.r.t $H$. Specifically, $\sigma_H$ is an isometry of $\mathbb{B}^N$, satisfies $\sigma_H^2 = Id$, $\sigma_H(x) = x$ for $x \in H$, and $\sigma_H H^{+}=H^{-}$ (see \cite{ADG}).

\medskip
\noindent
For a function $f: \mathbb{B}^N \to \mathbb{R}$, we define its polarization by $f^H: \mathbb{B}^N \to \mathbb{R}$ as follows:
\begin{equation*}
f^H(x)= \begin{cases}f(x), & \text { if } x \in H, \\ \max (f(x), f(\sigma_H(x))), & \text { if } x \in H^{+}, \\ \min (f(x), f(\sigma_H(x))), & \text { if } x \in H^{-}. \end{cases}
\end{equation*}

An initial step is to explore how polarization affects the non-local term and to analyze the equality cases.
\begin{lemma}\label{PolId}
Let $\alpha \in (0, N)$, $u \in L^{\frac{2 N}{N+\alpha}}\left(\mathbb{B}^N\right)$, and $H \subset \mathbb{B}^N$ be a closed, totally geodesic hypersurface. If $u \geq 0$ and
\begin{equation*}
\int_{\bn} \left[(- \Delta_{\mathbb{B}^{N}})^{-\frac{\alpha}{2}}u(y)\right]u(y) \dvg(y)\geqslant \int_{\bn} \left[(- \Delta_{\mathbb{B}^{N}})^{-\frac{\alpha}{2}}u^{H}(y)\right]u^{H}(y) \dvg(y),
\end{equation*}
then $u^H$ must either be equal to $u$ or $u \circ \sigma_H$.
\end{lemma}
Before proceeding with the proof of this lemma, it is crucial to establish the following monotonicity result (see also \cite{LLW}):
\begin{lemma}\label{MonoGreen}
  The function $k_{\alpha,N}$, as defined in \eqref{GreenFunc}, is positive and radially decreasing with respect to the geodesic distance $\rho$.
\end{lemma}

\begin{proof}

For $N=2m+1$
\begin{equation*}
    \begin{aligned}
        k_{\alpha, 2m+1}(\rho)&= \frac{1}{\Gamma(\alpha/2)} \int_{0}^{\infty} p_{t,2m+1}(\rho)\frac{\mathrm{~d}t}{t^{1-\alpha/2}}\\
        &= \frac{1}{\Gamma(\alpha/2)} \int_{0}^{\infty} 2^{-m-1} \pi^{-m-1 / 2} t^{-\frac{1}{2}} e^{-\frac{(N-1)^2}{4} t}\left(-\frac{1}{\sinh \rho} \frac{\partial}{\partial \rho}\right)^m e^{-\frac{\rho^2}{4 t}}\frac{\mathrm{~d}t}{t^{1-\alpha/2}}\\
        &= \frac{1}{\Gamma(\alpha/2)}2^{-m-1} \pi^{-m-1 / 2}\int_{0}^{\infty} t^{\frac{\alpha-3}{2}}e^{-\frac{(N-1)^2}{4} t}\left(-\frac{1}{\sinh \rho} \frac{\partial}{\partial \rho}\right)^m e^{-\frac{\rho^2}{4 t}}\mathrm{~d}t.
    \end{aligned}
\end{equation*}
\noindent
To obtain the monotonicity result, we differentiate the expression above. Observe that
\begin{equation*}
\begin{aligned}
    \frac{\partial}{\partial \rho}\left[-\frac{1}{\sinh \rho} \frac{\partial}{\partial \rho}\right]^m e^{-\frac{\rho^2}{4 t}}&=(-\sinh \rho)\frac{-1}{\sinh \rho}\frac{\partial}{\partial \rho}\left[-\frac{1}{\sinh \rho} \frac{\partial}{\partial \rho}\right]^m e^{-\frac{\rho^2}{4 t}}\\
    & = -\sinh \rho \left[-\frac{1}{\sinh \rho} \frac{\partial}{\partial \rho}\right]^{m+1}e^{-\frac{\rho^2}{4 t}}.
\end{aligned}
\end{equation*}
Since the heat kernel is positive as per \eqref{HeatKerBounds}, we get:
\begin{equation*}
    \begin{aligned}
        &\frac{d}{d\rho} k_{\alpha,2m+1}(\rho)\\
        &= \frac{-2\pi \sinh \rho}{\Gamma(\alpha/2)} \left[2^{-(m+1)-1} \pi^{-(m+1)-1 / 2}\int_{0}^{\infty} t^{\frac{\alpha-3}{2}}e^{Nt}e^{-(m+1)^2t}\left(-\frac{1}{\sinh \rho} \frac{\partial}{\partial \rho}\right)^{m+1}e^{-\frac{\rho^2}{4 t}}\mathrm{~d}t\right]\\
        &= \frac{-2\pi \sinh \rho}{\Gamma(\alpha/2)}\int_{0}^{\infty} e^{Nt} p_{t,2m+3}(\rho)\frac{\mathrm{~d}t}{t^{1-\alpha/2}}<0.
    \end{aligned}
\end{equation*}
Using the same approach, we find that $\frac{\partial}{\partial \rho} p_{t,2m+1}(\rho) < 0$.

\medskip
\noindent
For $N=2m$, applying Fubini's theorem, we obtain
\begin{equation*}
    \begin{aligned}
        &k_{\alpha, 2m}(\rho)\\
        &= \frac{1}{\Gamma(\alpha/2)} \int_{0}^{\infty} p_{t,2m}(\rho)\frac{\mathrm{~d}t}{t^{1-\alpha/2}}\\
&= \frac{1}{\Gamma(\alpha/2)} \int_{0}^{\infty} (2 \pi)^{-m-\frac{1}{2}} t^{-\frac{1}{2}} e^{-\frac{(N-1)^2}{4} t} \int_\rho^{+\infty} \frac{\sinh r}{\sqrt{\cosh r-\cosh \rho}}\left(-\frac{1}{\sinh r} \frac{\partial}{\partial r}\right)^m e^{-\frac{r^2}{4 t}} \mathrm{~d} r \frac{\mathrm{~d}t}{t^{1-\alpha/2}}\\
&= \frac{1}{\Gamma(\alpha/2)}\int_\rho^{+\infty}\frac{\sinh r}{\sqrt{\cosh r-\cosh \rho}}\int_{0}^{\infty} (2 \pi)^{-m-\frac{1}{2}}t^{-\frac{1}{2}} e^{-\frac{(N-1)^2}{4} t} \left(-\frac{1}{\sinh r} \frac{\partial}{\partial r}\right)^m e^{-\frac{r^2}{4 t}}\frac{\mathrm{~d}t}{t^{1-\alpha/2}}\mathrm{~d} r\\
&=\frac{\sqrt{2}}{\Gamma(\alpha/2)}\int_{0}^{\infty}
\frac{e^{\frac{t}{4}(2N-1)}}{t^{1-\alpha/2}}\int_\rho^{+\infty}\frac{\sinh r}{\sqrt{\cosh r-\cosh \rho}}p_{t,2m+1}(r)\mathrm{~d} r\mathrm{~d} t.
\end{aligned}
        \end{equation*}
        Making a change of variables, let $w=\sqrt{\cosh r-\cosh \rho}$, then
        \begin{equation*}
\frac{d}{d \rho} p_{t, 2 m+1}(r)=\frac{d p_{t, 2 m+1}(r)}{d r} \frac{\partial r}{\partial \rho}=\frac{d p_{t, 2 m+1}(r)}{d r} \frac{\sinh \rho}{\sinh r}<0 .
\end{equation*}
Thus,
\begin{equation*}
\frac{d}{d \rho} k_{\alpha, 2 m}(\rho)=\frac{2\sqrt{2}}{\Gamma(\alpha/2)}\int_{0}^{\infty}
\frac{e^{\frac{t}{4}(2N-1)}}{t^{1-\alpha/2}}\int_0^{\infty} \frac{d}{d \rho}p_{t, 2 m+1}(r(w, \rho)) \mathrm{~d} w\mathrm{~d} t<0.
\end{equation*}
With this, we achieve the desired result, thus completing the proof.
\end{proof}

\begin{proof}[\textbf{Proof of Lemma \ref{PolId}}]
Consider the following
\begin{equation*}
\begin{aligned}
&\int_{\bn} \left[(- \Delta_{\mathbb{B}^{N}})^{-\frac{\alpha}{2}}u(y)\right]u(y) \dvg(y)\\&= \int_{\bn} \left[k_{\alpha,N}*u   \right](y)\;u(y) \dvg(y)\\
&=\int_{\bn} \int_{\bn} \left[k_{\alpha,N}(d(T_{y}(x),0))u(x) \dvg(x)  \right]u(y) \dvg(y)\\
&= \int_{\bn} \int_{\bn} k_{\alpha,N}(d(T_{y}(x),0))u(x)u(y) \dvg(x) \dvg(y)\\
& =\int_{H^+} \int_{H^{+}} k_{\alpha,N}(d(T_{y}(x),0))\left[u(x) u(y)+u\left(\sigma_H(x)\right) u\left(\sigma_H(y)\right)\right] \\
& \hspace{2cm}+k_{\alpha,N}(d(T_{\sigma_{H}(y)}(x),0))\left[u(x) u\left(\sigma_H(y)\right)+u\left(\sigma_H(x)\right) u(y)\right] \dvg(x) \dvg(y),
\end{aligned}
\end{equation*}
and similarly for $u^H$. We aim to establish that the following inequality holds for all $x \in H^+$ and $y \in H^+$:
\begin{equation}
\begin{aligned}
&k_{\alpha,N}(d(T_{y}(x),0))\left[u(x) u(y)+u\left(\sigma_H(x)\right) u\left(\sigma_H(y)\right)\right] \\&\hspace{3cm}+k_{\alpha,N}(d(T_{\sigma_{H}(y)}(x),0))\left[u(x) u\left(\sigma_H(y)\right)+u\left(\sigma_H(x)\right) u(y)\right]\\
&\leqslant k_{\alpha,N}(d(T_{y}(x),0))\left[u^H(x) u^H(y)+u^H\left(\sigma_H(x)\right)u^H\left(\sigma_H(y)\right)\right]\\
& \hspace{3cm}+k_{\alpha,N}(d(T_{\sigma_{H}(y)}(x),0))\left[u^H(x) u^H\left(\sigma_H(y)\right)+u^H\left(\sigma_H(x)\right) u^H(y)\right].
\end{aligned}\label{PolIdeint}
\end{equation}
\noindent
We will now analyze each case individually to verify this inequality.

\medskip
\noindent
   \textbf{Case I} $u(\sigma_{H}(x))\leq u(x),\;\; u(\sigma_{H}(y))\leq u(y)$.\\
   In this case, we have
    \begin{equation*}
        u^H(x)= u(x), u^{H}(y)=u(y),u^H(\sigma_{H}(x))=u(\sigma_{H}(x)), u^H(\sigma_{H}(y))=u(\sigma_{H}(y)).
    \end{equation*}
   Consequently, equality holds in \eqref{PolIdeint}.
   
\medskip   
\noindent   
   \textbf{Case II} $u(\sigma_{H}(x))\geq u(x),\;\; u(\sigma_{H}(y))\geq u(y)$.\\
    This case can be handled in a manner similar to the previous one.
    
\medskip   
\noindent   
   \textbf{Case III} $u(x)\leq u(\sigma_{H}(x)),\;\; u(\sigma_{H}(y))\leq u(y).$\\
   Here, we have
    \begin{equation*}
        u^H(x)= u(\sigma_{H}(x)), u^{H}(y)=u(y),u^H(\sigma_{H}(x))=u(x), u^H(\sigma_{H}(y))=u(\sigma_{H}(y)).
    \end{equation*}
Moreover, given that $u \geq 0$, the subsequent inequality is satisfied
\begin{gather}
        \left[u(\sigma_{H}(x))-u(x)\right]u(y) \geq \left[u(\sigma_{H}(x))-u(x)\right]u(\sigma_{H}(y))\notag\\
        u(\sigma_{H}(x))u(y)+u(x)u(\sigma_{H}(y))\geq u(\sigma_{H}(x))u(\sigma_{H}(y))+u(x)u(y).\label{RearrIneq}\\ \notag
        \end{gather}
Now consider the LHS of \eqref{PolIdeint}
\noindent
\begin{equation*}
        \begin{aligned}
            &k_{\alpha,N}(d(T_{y}(x),0))\left[u(x) u(y)+u\left(\sigma_H(x)\right) u\left(\sigma_H(y)\right)\right]\\& \hspace{0.5cm}+k_{\alpha,N}(d(T_{\sigma_{H}(y)}(x),0))\left[u(x) u\left(\sigma_H(y)\right)+u\left(\sigma_H(x)\right) u(y)\right]\\\\
            &= k_{\alpha,N}(d(T_{\sigma_{H}(y)}(x),0))\left[u(x) u(y)+u\left(\sigma_H(x)\right) u\left(\sigma_H(y)\right)+u(x) u\left(\sigma_H(y)\right)+u\left(\sigma_H(x)\right) u(y)\right]\\
            &\hspace{0.5cm}+ \left[k_{\alpha,N}(d(T_{y}(x),0))-k_{\alpha,N}(d(T_{\sigma_{H}(y)}(x),0))\right]\left(u(x) u(y)+u\left(\sigma_H(x)\right) u\left(\sigma_H(y)\right)\right)\\\\
            &\leq k_{\alpha,N}(d(T_{\sigma_{H}(y)}(x),0))\left[u(x) u(y)+u\left(\sigma_H(x)\right) u\left(\sigma_H(y)\right)+u(x) u\left(\sigma_H(y)\right)+u\left(\sigma_H(x)\right) u(y)\right]\\
            &\hspace{0.5cm}+ \left[k_{\alpha,N}(d(T_{y}(x),0))-k_{\alpha,N}(d(T_{\sigma_{H}(y)}(x),0))\right]\left(u(\sigma_{H}(x))u(y)+u(x)u(\sigma_{H}(y))\right)\\\\
            &= k_{\alpha,N}(d(T_{\sigma_{H}(y)}(x),0))\left[u(x) u(y)+u\left(\sigma_H(x)\right) u\left(\sigma_H(y)\right)\right]\\
            &\hspace{0.5cm}+k_{\alpha,N}(d(T_{y}(x),0))\left[u(\sigma_{H}(x))u(y)+u(x)u(\sigma_{H}(y))\right]\\\\
           &=k_{\alpha,N}(d(T_{y}(x),0))\left[u^H(x)u^H(y)+u^H(\sigma_{H}(x))u^H(\sigma_{H}(y))\right] \\
            &\hspace{0.5cm}+ k_{\alpha,N}(d(T_{\sigma_{H}(y)}(x),0))\left[u^H(\sigma_{H}(x))u^H(y)+u^H(x)u^H(\sigma_{H}(y))\right].
        \end{aligned}
    \end{equation*}
    where we used the fact that $x, y \in H^+$, along with Lemma \ref{MonoGreen} and the rearrangement inequality \eqref{RearrIneq}.
    
    \medskip   
\noindent   
\textbf{Case IV}  $u(\sigma_{H}(x))\leq u(x),\;\;u(y)\leq  u(\sigma_{H}(y))$.\\
   This scenario can be addressed in a manner similar to Case III.

   \smallskip
   \noindent
Therefore, the lemma can be concluded utilizing \eqref{PolIdeint} and the provided hypothesis.
\end{proof}

Now we will showcase how the condition that either $u^H = u$ or $u^H = u \circ \sigma_H$ can be leveraged to infer certain symmetry properties.
\begin{lemma}\label{symLem}
For any $s \geq 1$ and $u \in L^s(\mathbb{B}^N)$, if $u \geq 0$ and for every closed, totally geodesic hypersurface $H \subset \mathbb{B}^N$, it holds that $u^H = u$ or $u^H = u \circ \sigma_H$, then there exist $x_0 \in \mathbb{B}^N$ and a non-increasing function $v: (0, \infty) \rightarrow \mathbb{R}$ such that for almost every $x \in \mathbb{B}^N$, $u(x) = v(d(x, x_0))$.
\end{lemma}

To establish Lemma \ref{symLem}, we begin by proving the following result (refer also to \cite{JSch1}).

\begin{lemma} \label{aidlemma}
Let $u \in L^s(\mathbb{B}^N)$ and $w \in L^t(\mathbb{B}^N)$ with $\frac{1}{s} + \frac{1}{t} = 1$ be a radial function such that for every $x, y \in \mathbb{B}^N$ with $d(x, 0) \leq d(y, 0)$, $w(x) \geq w(y)$, with equality if and only if $d(x, 0) = d(y, 0)$. Let $H \subset \mathbb{B}^N$ be a closed, totally geodesic hypersurface. If $y_0 \in H^+$ and
\begin{equation*}
\int_{\mathbb{B}^N} u^H(x) w(T_{y_0}(x)) \leq \int_{\mathbb{B}^N} u(x)w(T_{y_0}(x)), 
\end{equation*}
then $u = u^H$.
\end{lemma}
\begin{proof}
We will divide the proof into the following two distinct cases for $x \in H^+$.

\smallskip
\noindent
\textbf{Case I} $u(x) \leq u(\sigma_H(x))$. In this case, we obtain
\begin{equation*}
        u^H(x)= u(\sigma_{H}(x)),u^H(\sigma_{H}(x))=u(x).
    \end{equation*}
Additionally, for any $x \in H^+$, we observe that since $y_0 \in H^+$, $d(\sigma_H(x), y_0) > d(x, y_0)$, or equivalently, $d(T_{y_0}(\sigma_H(x)), 0) > d(T_{y_0}(x), 0)$. Thus, $w(T_{y_0}(x)) > w(T_{y_0}(\sigma_H(x)))$.

\noindent
\medskip
Now, consider
\begin{align*}
    [u^{H}(x)-u(x)]w(T_{y_0}(x))&=[u(\sigma_H(x))-u^H(\sigma_H(x))] w(T_{y_0}(x))\\
    &\geq [u(\sigma_H(x))-u^H(\sigma_H(x))] w\left(T_{y_0}(\sigma_{H}(x)\right),
    \end{align*}
   which leads to the inequality
\begin{equation}
u(x) w(T_{y_0}(x))+u\left(\sigma_H(x)\right) w\left(T_{y_0}(\sigma_{H}(x)\right) \leq u^H(x) w(T_{y_0}(x))+u^H\left(\sigma_H(x)\right) w(T_{y_0}(\sigma_{H}(x)).  \label{PolIneq1}
\end{equation}
\textbf{Case II} $u(x) \geq u(\sigma_H(x))$.\\
In this case, we arrive at
\begin{equation*}
u(x) w(T_{y_0}(x))+u\left(\sigma_H(x)\right) w\left(T_{y_0}(\sigma_{H}(x)\right)= u^H(x) w(T_{y_0}(x))+u^H\left(\sigma_H(x)\right) w\left(T_{y_0}(\sigma_{H}(x)\right). 
\end{equation*}
Thus, inequality \eqref{PolIneq1} holds in either case. By integrating this inequality over $H^+$, we obtain
\begin{equation}
\int_{\mathbb{B}^N} u^H(x) w(T_{y_0}(x)) \geq \int_{\mathbb{B}^N} u(x)w(T_{y_0}(x)). \label{PolIneqMain}
\end{equation}
Hence, based on the given assumption, we have for almost every $x \in H^+$
\begin{equation}
u(x) w(T_{y_0}(x))+u\left(\sigma_H(x)\right) w\left(T_{y_0}(\sigma_{H}(x)\right)= u^H(x) w(T_{y_0}(x))+u^H\left(\sigma_H(x)\right) w\left(T_{y_0}(\sigma_{H}(x)\right). \label{PolyEq}
\end{equation}
Since $y_0 \in H^+$, $d(\sigma_H(x), y_0) > d(x, y_0)$, and noting that $w$ is radially decreasing, we have $w(T_{y_0}(x)) > w\left(T_{y_0}(\sigma_H(x))\right)$. Therefore, it must be that $u^H(x) = u(x)$ and $u^H(\sigma_H(x)) = u(\sigma_H(x))$, because otherwise, we would fall under Case I and derive a strict inequality in \eqref{PolIneq1}, contradicting \eqref{PolyEq}.
\end{proof}

\begin{lemma}\label{invCha}
Let $u \in L^s(\mathbb{B}^N)$. If for every closed, totally geodesic hypersurface $H \subset \mathbb{B}^N$ with $x_0 \in H^+$, we have $u^H = u$, then there exists a non-increasing function $v: (0, \infty) \rightarrow \mathbb{R}$ such that $u(x) = v(d(x, x_0))$ for almost every $x \in \mathbb{B}^N$.
\end{lemma}
For a function $u: \mathbb{B}^N \rightarrow \mathbb{R}^{+} \cup {+\infty}$, the symmetric rearrangement, also known as Schwarz symmetrization, $u^*: \mathbb{B}^N \rightarrow \mathbb{R}^{+} \cup {+\infty}$, is defined as the unique function such that for every $\lambda > 0$, there exists $R \geq 0$ satisfying
\begin{equation*}
\left\{x \in \mathbb{B}^N : u^*(x) > \lambda\right\} = B_R(0),
\end{equation*}
and
\begin{equation*}
\mu\left\{x \in \mathbb{B}^N : u^*(x) > \lambda\right\} = \mu\left\{x \in \mathbb{B}^N : u(x) > \lambda\right\},
\end{equation*}
where $B_R(0)$ denotes the geodesic ball in $\mathbb{B}^N$ centered at $0$ with radius $R$, and $\mu$ denotes the volume measure in $\mathbb{B}^N$.

Thus, $u^*$ is a radial and radially decreasing function whose superlevel sets have the same measure as those of $u$.

Clearly, Lemma \ref{invCha} follows from the subsequent lemma:
\begin{lemma}
Let $s \geq 1$ and $u \in L^s(\mathbb{B}^N)$ be non-negative. The following conditions are equivalent:
\begin{enumerate}
\item $u^* = u$.
\item For almost every $x, y \in \mathbb{B}^N$, if $d(x,0) \leq d(y,0)$, then $u(x) \geq u(y)$.
\item For every closed, totally geodesic hypersurface $H$ such that $0 \in H^+$, $u(x) \geq u(\sigma_H(x))$, for almost every $x \in H^+$.
\item For every closed, totally geodesic hypersurface $H \subset \mathbb{B}^N$ such that $0 \in H^+$, $u^H = u$.
\end{enumerate}
\end{lemma}
\begin{proof}
The implication $\textit{(1)} \Rightarrow \textit{(2)}$ is straightforward. To prove $\textit{(2)} \Rightarrow \textit{(1)}$, since $u$ is radially decreasing, there exists a decreasing function $v: [0, +\infty) \rightarrow \mathbb{R}$ such that $u(x) = v(d(x,0))$. Given that $u \in L^s(\mathbb{B}^N)$, it is evident that $\lim_{r \rightarrow +\infty} v(r) = 0$. Hence, $u \geq 0$ and for all $\lambda > 0$, $\mu\left\{x \in \mathbb{B}^N : u^*(x) > \lambda\right\} = \mu\left\{x \in \mathbb{B}^N : u(x) > \lambda\right\}$, which implies $u = u^*$.

\medskip
There is nothing to prove for $\textit{(2)} \Rightarrow \textit{(3)}$. To prove $\textit{(3)} \Rightarrow \textit{(2)}$, let $x, y \in \mathbb{B}^N$ be such that $x \neq y$ and $d(x,0) \leq d(y,0)$. There exists a closed, totally geodesic hypersurface $H$ such that $x \in H^+$ and $y = \sigma_H(x)$. By assumption, we have $u(y) = u(\sigma_H(x)) \leq u(x)$.

\medskip
Finally, the equivalence $\textit{(3)} \iff \textit{(4)}$ follows from the definition of $u^H$.
\end{proof}

\begin{proof}[\textbf{Proof of Lemma \ref{symLem}}] Select $w \geq 0$ satisfying the assumptions of Lemma \ref{aidlemma}. For $x \in \mathbb{B}^N$, define the function 
\begin{equation*}
W(x)=\int_{\bn} u(y) w(T_{x}(y))\dvg(y).
\end{equation*}

This function is nonnegative, continuous, and $\lim_{d(x,0) \rightarrow \infty} W(x) = 0$, indicating it attains a maximum at some point $x_0 \in \bn$. 

To finalize the proof using Lemma \ref{invCha}, we need to show that for every closed, totally geodesic hypersurface $H \subset \mathbb{B}^N$ where $x_0$ is in $H^+$, the equality $u^H = u$ holds. Specifically, if $u^H = u \circ \sigma_H$, then, considering that $w$ is radial and based on the definition of $x_0$,
\begin{equation*}
\begin{aligned}
\int_{\bn} u^H(y) w\left(T_{x_0}(y)\right) \dvg(y) & =\int_{\bn} u\left(\sigma_H(y)\right) w\left(T_{x_0}(y)\right)\dvg(y) \\
& =\int_{\bn} u(y) w\left(T_{\sigma_H(x_0)}(y)\right)\dvg(y)\\
& \leqslant \int_{\bn} u(y) w\left(T_{x_0}(y)\right)\dvg(y).
\end{aligned}
\end{equation*}
Therefore, by Lemma \ref{aidlemma}, we have $u^H = u$.
\end{proof}
\begin{proof}[\textbf{Proof of Proposition \ref{RadSymm}}]
Consider a closed, totally geodesic hypersurface $H \subset \mathbb{B}^N$. Observe the following equalities:
\begin{equation*}
\int_{\mathbb{B}^N} \left|\nabla_{\bn} u^H\right|^2 = \int_{\mathbb{B}^N} \left|\nabla_{\bn} u\right|^2 \text{and}
\int_{\mathbb{B}^N} \left|u^H\right|^2 = \int_{\mathbb{B}^N} \left|u\right|^2.
\end{equation*}
Based on the characterization of ground states, it follows that
\begin{equation*}
\int_{\mathbb{B}^N} \left[(- \Delta_{\mathbb{B}^N})^{-\frac{\alpha}{2}} |u|^p(y)\right] |u|^p(y) \dvg(y) \geq \int_{\mathbb{B}^N} \left[(- \Delta_{\mathbb{B}^N})^{-\frac{\alpha}{2}} |u^H|^p(y)\right] |u^H|^p(y) \dvg(y).
\end{equation*}

According to Lemma \ref{PolId}, for any closed, totally geodesic hypersurface $H \subset \mathbb{B}^N$, it holds that either $u^H = u$ or $u^H = u \circ \sigma_H$. Therefore, the conclusion can be drawn from Lemma \ref{symLem}.

\end{proof}
\end{section}
\begin{section}{Regularity}\label{regularity}
In this section, we examine the regularity of the solution to the problem \eqref{mainEq} following the approach of \cite{MS2}. To do this, we require a nonlocal analog of the estimate provided in \cite[Lemma 2.1]{BK} by Brezis and Kato, which can be established using the inequality stated in Lemma \ref{BKineqEss}. In particular, we prove the following result in this section.
\begin{theorem}\label{RegThm}
    Let $N \geq 3$, $\lambda < \frac{(N-1)^2}{4}$and $\alpha \in(0, N)$. If $u$ is a solution of \eqref{mainEq} for $\frac{N+\alpha}{N} \leq p \leq \frac{N+\alpha}{N-2}$, then $u \in L^q(\bn)$ for every $q \in\left[2, \frac{N}{\alpha} \frac{2 N}{N-2}\right)$. Moreover, if $2<p\leq \frac{N+\alpha}{N-2}$ and $1\leq \alpha <N$, then $u \in L^{\infty}(\bn)$.
    \end{theorem}
    To establish the preceding theorem, we need the following auxiliary lemmas.
\begin{lemma} \label{BrezisKatoIneq}
    Let $N \geq 2, \alpha \in(0,N)$ and $\theta \in(0,2)$. If $H, K \in L^{\frac{2 N}{\alpha+2}}\left(\mathbb{B}^{N}\right)+$ $L^{\frac{2 N}{\alpha}}\left(\mathbb{B}^{N}\right)$ and $\frac{\alpha}{N}<\theta<2-\frac{\alpha}{N}$, then for every $\varepsilon>0$, there exists $C_{\varepsilon, \theta} \in \mathbb{R}$ such that for every $u \in H^{1}\left(\mathbb{B}^{N}\right)$,
\noindent
$$
\int_{\mathbb{B}^{N}}\left((- \Delta_{\mathbb{B}^{N}})^{-\frac{\alpha}{2}}\left(H|u|^{\theta}\right)\right) K|u|^{2-\theta}\dvg \leq \varepsilon^{2} \int_{\mathbb{B}^{N}}|\nabla_{\bn} u|^{2}\dvg+C_{\varepsilon, \theta}\int_{\mathbb{B}^{N}}|u|^{2}\dvg.
$$
\end{lemma}

\begin{lemma}\label{BKineqEss}
Let $N\geq 2$ and $0<\alpha<N$. Suppose $q, r, s, t \in[1, \infty)$ and $\tau \in[0,2]$ satisfy the following conditions:
    \begin{equation*}
1+\frac{\alpha}{N}-\frac{1}{s}-\frac{1}{t}=\frac{\tau}{q}+\frac{2-\tau}{r}
\end{equation*}
If $\theta \in(0,2)$ satisfies
\noindent
$$
\begin{gathered}
\min (q, r)\left(\frac{\alpha}{N}-\frac{1}{s}\right)<\theta<\max (q, r)\left(1-\frac{1}{s}\right), \\
\min (q, r)\left(\frac{\alpha}{N}-\frac{1}{t}\right)<2-\theta<\max (q, r)\left(1-\frac{1}{t}\right),
\end{gathered}
$$
\noindent
then for every $H \in L^{s}\left(\mathbb{B}^{N}\right)$, $K \in L^{t}\left(\mathbb{B}^{N}\right)$ and $u \in L^{q}\left(\mathbb{B}^{N}\right) \cap L^{r}\left(\mathbb{B}^{N}\right)$,
$$\int_{\mathbb{B}^{N}}\left((- \Delta_{\mathbb{B}^{N}})^{-\frac{\alpha}{2}}\left(H|u|^{\theta}\right)\right) K|u|^{2-\theta} \leq C\|H\|_{L^s(\bn)} \|K\|_{L^t(\bn)} \|u\|_{L^q(\bn)}^{\tau}\|u\|_{L^r(\bn)}^{2-\tau}.$$
\end{lemma}
The proofs of the above two lemmas follow exactly on the same lines as \cite[Lemma 3.2, 3.3]{MS2}, utilizing the Hölder and HLS inequalities repeatedly. Consequently, we omit the proofs here. 
\begin{proposition} \label{ExtLp}
     If $H, K \in L^{\frac{2 N}{\alpha}}\left(\mathbb{B}^{N}\right)+L^{\frac{2 N}{\alpha+2}}\left(\mathbb{B}^{N}\right)$, $\lambda<\frac{(N-1)^2}{4}$, and $u \in H^{1}\left(\mathbb{B}^{N}\right)$ solves
\begin{equation*}
-\Delta_{\bn}u-\lambda u=\left((- \Delta_{\mathbb{B}^{N}})^{-\frac{\alpha}{2}} H u\right) K
\end{equation*}
then $u \in L^{p}\left(\mathbb{B}^{N}\right)$ for every $p \in\left[2, \frac{N}{\alpha} \frac{2 N}{N-2}\right)$.
\end{proposition}

\begin{proof}
    By setting $\theta = 1$ in Lemma \ref{BrezisKatoIneq}, there exists a constant $\beta > 0$ such that $\beta > 2\lambda$ and

\begin{equation}
\int_{\mathbb{B}^{N}}\left((- \Delta_{\mathbb{B}^{N}})^{-\frac{\alpha}{2}}|H \phi|\right) |K\phi|  \leq \frac{1}{2} \int_{\mathbb{B}^{N}}|\nabla_{\bn} \phi|^{2} \dvg+\frac{\beta}{2} \int_{\mathbb{B}^{N}}|\phi|^{2} \dvg \label{BKuse}
\end{equation}
for all $\phi \in H^{1}(\mathbb{B}^{N})$. Consider the sequences $\left\{H_{k}\right\}_{k \in \mathbb{N}}$ and $\left\{K_{k}\right\}_{k \in \mathbb{N}}$ in $L^{\frac{2N}{\alpha}}(\mathbb{B}^{N})$ such that $\left|H_{k}\right| \leq H$, ${H_{k}(x)} \to H(x)$ almost everywhere, and $\left|K_{k}\right| \leq K$, ${K_{k}(x)} \to K(x)$ almost everywhere in $\mathbb{B}^N$. For each $k \in \mathbb{N}$, define $a_{k} \colon H^{1}(\mathbb{B}^{N}) \times H^{1}(\mathbb{B}^{N}) \to \mathbb{R}$ such that for $w, v \in H^{1}(\mathbb{B}^{N})$
\begin{align*}
a_{k}(w, v):= & \int_{\mathbb{B}^{N}} \nabla_{\bn} w \cdot \nabla_{\bn} v - \lambda  \int_{\mathbb{B}^{N}} w v+\beta \int_{\mathbb{B}^{N}} w v - \int_{\mathbb{B}^{N}}\left((- \Delta_{\mathbb{B}^{N}})^{-\frac{\alpha}{2}} H_{k} w\right) K_{k} v.
\end{align*}
Using \eqref{firsteigen} and \eqref{BKuse}, we obtain
\begin{align*}
&a_{k}(v, v)\\
&\geq \int_{\mathbb{B}^{N}}|\nabla_{\bn} v|^{2} \dvg+\left(\beta-\lambda\right) \int_{\mathbb{B}^{N}}|v|^{2} \dvg -\left(\frac{1}{2} \int_{\mathbb{B}^{N}}|\nabla_{\bn} v|^{2} \dvg+\frac{\beta}{2} \int_{\mathbb{B}^{N}}|v|^{2} \dvg\right)\\
&\geq \;\frac{1}{2} \int_{\mathbb{B}^{N}}|\nabla_{\bn} v|^{2} \dvg+\left(\frac{\beta}{2}-\lambda\right)\int_{\mathbb{B}^{N}}|v|^{2} \dvg \geq C \|v\|_{\lambda}^2.
\end{align*}
Furthermore, by applying Hölder's inequality to the last term of $a_k(w,v)$, with the exponents chosen such that $\frac{1}{r_1} = \frac{1}{2} - \frac{\alpha}{2N}$ and $\frac{1}{r_2} = \frac{\alpha}{2N} + \frac{1}{2}$, and then employing the Hardy-Littlewood-Sobolev inequality, the boundedness of the bilinear form can be easily established. Consequently, $a_k$ is a bounded and coercive bilinear form. Therefore, by invoking the Lax-Milgram theorem for the functional $f: H^{1}\left(\mathbb{B}^{N}\right) \rightarrow \mathbb{R}$ defined as:
\noindent
$$
f(v)=\int_{\mathbb{B}^{N}}\beta\;u \;v \dvg \quad \forall \;v \in H^{1}\left(\mathbb{B}^{N}\right),
$$
\noindent
there exists unique $u_{k} \in H^{1}\left(\mathbb{B}^{N}\right)$ such that $a_{k}\left(u_{k}, v\right)=f(v)$ holds for every $v \in H^{1}\left(\mathbb{B}^{N}\right)$.  Consequently, $u_{k}$
satisfies the equation
\begin{equation}
-\Delta_{\bn} u_{k}-\lambda u_{k}+\beta u_{k}=\left((- \Delta_{\mathbb{B}^{N}})^{-\frac{\alpha}{2}} H_{k} u_{k}\right) K_{k}+\beta u \text { in } \mathbb{B}^{N}. \label{PertProb}
\end{equation}
It follows that the sequence $\left\{u_{k}\right\} \rightharpoonup u$ in $H^{1}\left(\mathbb{B}^{N}\right)$ as $k \rightarrow \infty$. 
For any $\gamma>0$, we define the truncated function $u_{k, \gamma}$ by
$$
u_{k, \gamma}(x)= \begin{cases}-\gamma & \text { if } u_{k}(x) \leq-\gamma, \\ u_{k}(x) & \text { if }-\gamma<u_{k}(x)<\gamma, \\ \gamma & \text { if } u_{k}(x) \geq \gamma.\end{cases}
$$

By choosing the test function $v(x)=\left|u_{k, \gamma}\right|^{q-2} u_{k, \gamma} \in H^{1}\left(\mathbb{B}^{N}\right)$ as a test function in \eqref{PertProb}, we get
$$
\begin{aligned}
& (q-1) \int_{\mathbb{B}^{N}}\left|u_{k, \gamma}\right|^{q-2} \nabla_{\bn} u_{k} \cdot \nabla_{\bn} u_{k, \gamma} \dvg+\left(\beta - \lambda \right)\int_{\mathbb{B}^{N}}\left|u_{k, \gamma}\right|^{q-2} u_{k} u_{k, \gamma} \dvg \\
& = \int_{\mathbb{B}^{N}}\left((- \Delta_{\mathbb{B}^{N}})^{-\frac{\alpha}{2}} H_{k} u_{k}\right) K_{k}\left|u_{k, \gamma}\right|^{q-2} u_{k, \gamma} \dvg+ \beta \int_{\mathbb{B}^{N}} u \left|u_{k, \gamma}\right|^{q-2} u_{k, \gamma} \dvg .
\end{aligned}
$$
Additionally, by the construction of $u_{k, \gamma}$, we have $\left|u_{k, \gamma}\right|^{q} \leq\left|u_{k, \gamma}\right|^{q-2} u_{k, \gamma} u_{k}$. Furthermore, using the relation $\left|\nabla_{\bn} u(x)\right|_{\bn}^2=\left(\frac{1-|x|^2}{2}\right)^2 \left|\nabla u(x)\right|^2$, we get

\begin{align*}
& (q-1) \int_{\mathbb{B}^{N}}\left|u_{k, \gamma}\right|^{q-2}\left|\nabla_{\bn} u_{k, \gamma}\right|^{2} \dvg+\left(\beta-\lambda\right)\int_{\mathbb{B}^{N}}\left|u_{k, \gamma}\right|^{q-2} u_{k, \gamma} u_{k} \dvg \\
& \geq \frac{4(q-1)}{q^{2}} \int_{\mathbb{B}^{N}}\left|\nabla_{\bn}\left(u_{k, \gamma}\right)^{\frac{q}{2}}\right|^{2} \dvg+\left(\beta-\lambda\right) \int_{\mathbb{B}^{N}}\left|\left| u_{k, \gamma}\right|^{\frac{q}{2}}\right|^{2} \dvg.
\end{align*}
Thus, we have
\begin{align*}
& \frac{4(q-1)}{q^{2}} \int_{\mathbb{B}^{N}}\left|\nabla_{\bn}\left(u_{k, \gamma}\right)^{\frac{q}{2}}\right|^{2} \dvg+\left(\beta-\lambda\right) \int_{\mathbb{B}^{N}}\left|\left| u_{k, \gamma}\right|^{\frac{q}{2}}\right|^{2} \dvg \\
\leq & \;(q-1) \int_{\mathbb{B}^{N}}\left|u_{k, \gamma}\right|^{q-2} \nabla_{\bn} u_{k, \gamma} \cdot \nabla_{\bn} u_{k} \dvg+\left(\beta-\lambda\right) \int_{\mathbb{B}^{N}}\left|u_{k, \gamma}\right|^{q-2} u_{k, \gamma} u_{k} \dvg \\
= &  \int_{\mathbb{B}^{N}}\left((- \Delta_{\mathbb{B}^{N}})^{-\frac{\alpha}{2}} H_{k} u_{k}\right) K_{k}\left|u_{k, \gamma}\right|^{q-2} u_{k, \gamma} \dvg+ \beta \int_{\mathbb{B}^{N}} u \left|u_{k, \gamma}\right|^{q-2} u_{k, \gamma} \dvg.
\end{align*}
For $2\leq q<\frac{2 N}{\alpha}$, selecting $\theta=\frac{2}{q} \in \left(\frac{\alpha}{N}, 2-\frac{\alpha}{N}\right)$ and $u=u_{k, \gamma}$ in Lemma \ref{BrezisKatoIneq}, we obtain
\begin{align*}
&  \int_{\mathbb{B}^{N}}\left((- \Delta_{\mathbb{B}^{N}})^{-\frac{\alpha}{2}}H_{k} u_{k, \gamma}\right) K_{k}\left|u_{k, \gamma}\right|^{q-2} u_{k, \gamma} \leq \frac{2(q-1)}{q^{2}} \int_{\mathbb{B}^{N}}\left|\nabla_{\bn}\left(u_{k, \gamma}\right)^{\frac{q}{2}}\right|^{2}+C\int_{\mathbb{B}^{N}}\left|\left| u_{k, \gamma}\right|^{\frac{q}{2}}\right|^{2}.
\end{align*}
Now, setting $A_{k, \gamma}=\left\{x \in \mathbb{B}^{N}:\left|u_{k}(x)\right|>\gamma\right\}$, we derive
\begin{align*}
& \frac{4(q-1)}{q^{2}}\int_{\mathbb{B}^{N}}\left|\nabla_{\bn}\left(u_{k, \gamma}\right)^{\frac{q}{2}}\right|^{2} \dvg+\left(\beta - \lambda\right)\int_{\mathbb{B}^{N}}\left|\left| u_{k, \gamma}\right|^{\frac{q}{2}}\right|^{2} \dvg\\
&-C \int_{\mathbb{B}^{N}}\left|\left| u_{k, \gamma}\right|^{\frac{q}{2}}\right|^{2} \dvg -\frac{2(q-1)}{q^{2}}\int_{\mathbb{B}^{N}}\left|\nabla_{\bn}\left( u_{k, \gamma}\right)^{\frac{q}{2}}\right|^{2} \dvg \\
& \leq 2 \int_{A_{k, \gamma}}\left((- \Delta_{\mathbb{B}^{N}})^{-\frac{\alpha}{2}}\left|K_{k}\right|\left|u_{k}\right|^{q-1}\right)\left|H_{k} u_{k}\right|\dvg+\beta \int_{\mathbb{B}^{N}}\left|u_{k, \gamma}\right|^{q-2} u_{k, \gamma} u \dvg.
\end{align*}
This implies
\begin{equation}
\begin{aligned}
& \frac{2(q-1)}{q^{2}}\int_{\mathbb{B}^{N}}\left|\nabla_{\bn}\left( u_{k, \gamma}\right)^{\frac{q}{2}}\right|^{2} \dvg \\
& \leq 2  \int_{A_{k, \gamma}}\left((- \Delta_{\mathbb{B}^{N}})^{-\frac{\alpha}{2}}\left|K_{k}\right|\left|u_{k}\right|^{q-1}\right)\left|H_{k} u_{k}\right|\dvg+C\left[\int_{\mathbb{B}^{N}}\left|u_{k, \gamma}\right|^{q-2} u_{k, \gamma}u + \left|u_{k, \gamma}\right|^{q} \dvg\right] \\
& \leq 2  \int_{A_{k, \gamma}}\left((- \Delta_{\mathbb{B}^{N}})^{-\frac{\alpha}{2}}\left|K_{k}\right|\left|u_{k}\right|^{q-1}\right)\left|H_{k} u_{k}\right|\dvg+C\left[\int_{\mathbb{B}^{N}}\left|u_{k}\right|^{q}+|u|^{q}\dvg \right].
\end{aligned} \label{11a}
\end{equation}
Applying the Hardy-Littlewood-Sobolev inequality gives
\begin{equation*}
\int_{A_{k, \gamma}}\left((- \Delta_{\mathbb{B}^{N}})^{-\frac{\alpha}{2}}\left|K_{k}\right|\left|u_{k}\right|^{q-1}\right)\left|H_{k} u_{k}\right| \leq C\left(\int_{\mathbb{B}^{N}}\left|\left| K_{k}\right|\left| u_{k}\right|^{q-1}\right|^{r}\right)^{\frac{1}{r}}\left(\int_{A_{k, \gamma}}\left|H_{k} u_{k}\right|^{s}\right)^{\frac{1}{s}}
\end{equation*}
where $\frac{1}{r}=\frac{\alpha}{2N}+1-\frac{1}{q}$ and $\frac{1}{s}=\frac{\alpha}{2N}+\frac{1}{q}$. Assuming $u_{k} \in L^{q}\left(\mathbb{B}^{N}\right)$, it follows from Hölder's inequality that $\left|K_{k}\right|\left|u_{k}\right|^{q-1} \in L^{r}\left(\mathbb{B}^{N}\right)$ and $\left|H_{k} u_{k}\right| \in L^{s}\left(\mathbb{B}^{N}\right)$. As a result, we have
\begin{equation*}
 \int_{A_{k, \gamma}}\left((- \Delta_{\mathbb{B}^{N}})^{-\frac{\alpha}{2}}\left|K_{k}\right|\left|u_{k}\right|^{q-1}\right)\left|H_{k} u_{k}\right| \rightarrow 0 \text { as } \gamma \rightarrow \infty.
\end{equation*}
By taking the limit $\gamma \rightarrow \infty$ in \eqref{11a}, we derive
\begin{equation*}
\frac{2(q-1)}{q^{2}}\int_{\mathbb{B}^{N}}\left|\nabla_{\bn}\left( u_{k}\right)^{\frac{q}{2}}\right|^{2}\dvg \leq C\left(\int_{\mathbb{B}^{N}}\left|u_{k}\right|^{q}+|u|^{q}\dvg\right).
\end{equation*}
This indicates that $\left|u_{k}\right|^{\frac{q}{2}} \in H^{1}\left(\mathbb{B}^{N}\right) \hookrightarrow L^{\frac{2 N}{N-2}}\left(\mathbb{B}^{N}\right)$. Therefore, by Fatou's lemma
\noindent
$$
\begin{aligned}
\left(\int_{\mathbb{B}^{N}}\left|\left| u_{k}\right|^{\frac{q}{2}}\right|^{\frac{2 N}{N-2}}\dvg\right)^{\frac{2(N-2)}{2 N}} & \leq C\left\|\left(u_{k}\right)^\frac{q}{2}\right\|_\lambda^{2} \leq C\left(\int_{\mathbb{B}^{N}}\left|u_{k}\right|^{q}+|u|^{q}\dvg\right) \\
& \leq C\left(\int_{\mathbb{B}^{N}}\left|u_{k}\right|^{q}\dvg+\limsup_{k \rightarrow \infty} \int_{\mathbb{B}^{N}}\left|u_{k}\right|^{q}\dvg\right).
\end{aligned}
$$
\noindent
Thus, we conclude
\noindent
$$
\limsup _{k \rightarrow \infty}\left(\int_{\mathbb{B}^{N}}\left|u_{k}\right|^{\frac{N q}{N-2}}\dvg\right)^{1-\frac{2}{N}} \leq C \limsup _{k \rightarrow \infty} \int_{\mathbb{B}^{N}}\left|u_{k}\right|^{q}\dvg.
$$
\noindent
By iterating the above argument finitely many times on $q$, we obtain $u \in L^{r}\left(\mathbb{B}^{N}\right)$ for all $r \in\left[2, \frac{2 N}{\alpha} \frac{N}{N-2}\right)$. 
\end{proof}

\begin{proof}[Proof of Theorem \ref{RegThm}]
 Let $u(x)$ be a solution of \eqref{mainEq}. Since $\frac{N+\alpha}{N} \leq p \leq \frac{N+\alpha}{N-2}$, we obtain
 \noindent
$$
|u(x)|^{p-1} \leq\left(|u(x)|^{\frac{\alpha}{N}}+|u(x)|^{\frac{\alpha+2}{N-2}}\right).
$$
\noindent
Setting $H(x)=K(x)=|u(x)|^{p-2} u(x)$, it follows that $H, K \in L^{\frac{2 N}{\alpha}}\left(\mathbb{B}^{N}\right)+L^{\frac{2 N}{\alpha+2}}\left(\mathbb{B}^{N}\right)$. Therefore, by Proposition \ref{ExtLp}, we conclude that $u \in L^{q}\left(\mathbb{B}^{N}\right)$ for $q \in\left[2, \frac{N}{\alpha} \frac{2 N}{N-2}\right)$.

We will now demonstrate that $u \in L^{\infty}(\bn)$.  To establish this, we begin by claiming that  $(-\Delta_{\mathbb{B}^{N}})^{-\frac{\alpha}{2}} |u|^{p} \in L^{\infty}\left(\bn\right)$. \\
To justify this claim, observe that for $\beta \in \mathbb{R}$, $R>0$, we have 
\begin{align*}
    \int_{B_R(0)}\frac{1}{(\sinh (d(x,0)))^\beta}\dvg(x)&= \omega_{N-1} \int_0^{R}\frac{(\sinh \rho)^{N-1}}{(\sinh \rho)^\beta} \mathrm{~d}\rho\\
    &\leq \omega_{N-1}\int_0^{R}\frac{(\sinh \rho)^{N-1}\cosh \rho}{(\sinh \rho)^\beta} \mathrm{~d}\rho,
\end{align*}
where $\omega_{N-1}$ is the volume of $\mathbb{S}^{N-1}$. Consequently, this integral converges when $\beta<N$ Similarly, if $\beta >N$, then $$\int_{\left(B_R(0)\right)^{\complement}}\frac{1}{(\sinh (d(x,0)))^\beta}\dvg(x) < \infty.$$ \\
Using the estimates from (4.1) in \cite{LLY}, we obtain:

\smallskip
\noindent
For $0<\rho<1$, $0<\alpha<N$
\begin{align*}
    k_{\alpha,N}(\rho)\leq C \frac{1}{\rho^{N-\alpha}}\leq C(N,\alpha) \frac{1}{(\sinh \rho)^{N-\alpha}}.
\end{align*}
Further, for $\rho\geq 1$ and $\alpha>2$ 
\begin{align*}
    k_{\alpha,N}(\rho)\leq C\rho^{\frac{\alpha-2}{2}} e^{-(N-1)\rho}\leq C\rho^{\alpha-2}e^{-(N-1)\rho}\leq C e^{-\rho(N-\alpha)}\leq C  \frac{1}{(\sinh \rho)^{N-\alpha}}.
\end{align*}
For $1\leq\alpha\leq2$, and $\rho \geq 1$, note that
\begin{align*}
    \frac{e^{-(N-1)\rho}(\sinh \rho)^{N-\alpha}}{\rho^{\frac{2-\alpha}{2}} }\leq \left(\frac{\sinh \rho}{e^\rho}\right)^{N}\frac{e^\rho}{(\sinh \rho)^\alpha}\leq \frac{1}{(\sinh \rho)^{\alpha-1}}\frac{e^\rho}{\sinh \rho}\leq C.
\end{align*}
Now, consider 
\begin{align*}
    (-\Delta_{\mathbb{B}^{N}})^{-\frac{\alpha}{2}} |u|^{p} = k_{\alpha,N} * |u|^p = \Bigl(\chi_{B_R(0)}k_{\alpha,N}\Bigr) * |u|^p+ \left(\chi_{(B_R(0))^{\complement}}k_{\alpha,N}\right) * |u|^p.
    \end{align*}
    From the estimates above, we can find $p_1,p_2$ such that 
    $$\chi_{B_R(0)}k_{\alpha,N} \in L^{p_1}(\bn) \text{ for some } p_1 \in \left(1, \frac{N}{N-\alpha}\right),$$
    and
    $$\chi_{(B_R(0))^\complement}k_{\alpha,N} \in L^{p_2}(\bn) \text{ for some } p_2 \in \left( \frac{N}{N-\alpha},\infty\right).$$
 Given that $\frac{N + \alpha}{N} \leq p \leq \frac{N + \alpha}{N - 2}$, it follows that $2 < \frac{N p}{\alpha} < \frac{N}{\alpha} \frac{2N}{N - 2}$. Therefore, there exist $q_1$ and $q_2$ such that \begin{align*} |u|^p \in L^{q_1}(\mathbb{B}^N) \cap L^{q_2}(\mathbb{B}^N), \end{align*} where $q_i$ and $p_i$ are conjugate indices for $i = 1, 2$. Applying Hölder's inequality, we establish the desired claim. We can conclude the theorem by applying Theorem 4.1 from \cite{BS}, noting that $2 < p \leq \frac{N + \alpha}{N - 2} < \frac{2N}{N - 2}$.
\end{proof}
\begin{remark}
    The $L^p$ Calderón–Zygmund inequality on non-compact Riemannian manifolds has been investigated in \cite{GP}. However, to achieve $W^{2,p}$-type regularity in the current setting, the results from \cite{GP} cannot be directly applied; additional assumptions on the solution are necessary (see \cite[Proposition 3.8]{GP}). 
\end{remark}
\end{section}
\textbf{Acknowledgments.} 
First author is supported by the Prime Minister's Research Fellowship (PMRF). A part of this research is supported by DST-FIST project No. SR/FST/MS-1/2019/45. Authors sincerely acknowledge D. Ganguly for the valuable discussions and insights.

\end{document}